\documentclass[11pt, reqno]{amsart}
\usepackage{fullpage}

\newif\ifArxiv
\Arxivtrue

 \newif\ifHideFoot
\HideFoottrue

\ifHideFoot
\else
\usepackage[notref,notcite]{showkeys}
\fi 

\usepackage{amssymb,amsmath,amsthm,amscd,mathrsfs,graphicx, color}
\usepackage[cmtip,all,matrix,arrow,tips,curve]{xy}
 \usepackage{hyperref}
\usepackage[usenames,dvipsnames]{xcolor}
\usepackage{xypic}
\hypersetup{colorlinks=true,citecolor=ForestGreen,linkcolor=Maroon,urlcolor=NavyBlue}
 \usepackage{mathtools}
\usepackage{mathpazo}
\usepackage[normalem]{ulem}
\usepackage{thmtools}
\usepackage{cleveref}
\usepackage{enumitem}
\usepackage{tikz-cd}
  
\numberwithin{equation}{section}
\newtheorem{teo}{Theorem}[section]
\newtheorem{pro}[teo]{Proposition}
\newtheorem{lem}[teo]{Lemma}

\newtheorem{teoalpha}{Theorem}

\newtheorem{coralpha}[teoalpha]{Corollary}

\theoremstyle{definition}
\newtheorem{exa}[teo]{Example}
\newtheorem{dfn}[teo]{Definition}
\theoremstyle{remark}
\newtheorem{rem}[teo]{Remark}

\ifHideFoot

\newcommand{\Yano}[1]{}
\newcommand{\Shend}[1]{}

\else 

\newcommand{\marg}[1]{\normalsize{{
\color{red}\footnote{{\color{blue}#1}}}{\marginpar[\vskip
-.25cm{\color{red}\hfill\tiny\thefootnote$\implies$}]{\vskip
-.2cm{\color{red}$\impliedby$\tiny\thefootnote}}}}}
\newcommand{\Yano}[1]{\marg{(Yano) #1}}
\newcommand{\Shend}[1]{\marg{(Shend) #1}}

\fi

\newcommand{\m}[1]{\mathcal{#1}}
\newcommand{\OO}{\mathcal{O}}
\newcommand{\X}{\mathcal{X}}

\title[Foliations, slope stability, and positivity]{Foliations, slope stability, and positivity of log canonical bundles on Deligne--Mumford stacks}

\author{Sebastian Casalaina-Martin}
\address{University of Colorado, Department of Mathematics, 
Boulder, CO 80309, USA }
\email{casa@math.colorado.edu}

\author{Shend Zhjeqi}
\address{University of Michigan, Department of Mathematics, 
Ann Arbor, MI 48109, USA }
\email{shendzh@umich.edu}

\thanks{Research of the first named author is supported in part by a grant from the Simons Foundation (SFI-MPS-TSM-00013682). The second named author was partially supported by the Simons Collaboration grant Moduli of Varieties.}

\date{\today}

\begin{document}

\begin{abstract}
We generalize some results of Campana--P\u{a}un regarding foliations, slope stability, and positivity of log canonical bundles on smooth projective varieties to the case of smooth proper DM stacks admitting projective coarse moduli spaces.  This paper is the third in a series aiming to generalize results of Popa--Schnell and Wei--Wu on Viehweg hyperbolicity to the setting of DM stacks, and in particular, to certain KSBA moduli spaces.
\end{abstract}

\maketitle

\section*{Introduction}

In \cite{CPFol19}, Campana--P\u aun showed that for a smooth projective variety $X$ with a reduced normal crossing divisor $\Delta\subseteq X$, if some positive tensor power of $\Omega^1_X(\log \Delta)$ contains a subsheaf with big determinant, then the first Chern class of every torsion-free quotient sheaf of every positive tensor power of $\Omega^1_X(\log \Delta)$ is pseudo-effective.  This has a number of applications in birational geometry, including to the birational geometry of moduli spaces.  With the end goal of strengthening some of the  applications to moduli spaces, in this paper, we extend the results of Campana--P\u aun to the setting of smooth proper Deligne--Mumford (DM) stacks over $\mathbb C$ with projective coarse moduli space.  We note that in their paper, Campana--P\u aun actually consider the special case where such a stack $\mathcal X$ is the quotient of a smooth projective variety by the action of a finite group; i.e., $\mathcal X=[V/G]$ where $V$ is a smooth projective variety and $G$ is a finite group acting on $V$.

  In this paper, we extend \cite[Thm.~7.6, Thm.~1.2]{CPFol19} to the more general class of DM stacks mentioned above:

\begin{teoalpha}

\label{TA:SonCP-T4}
 Let $\mathcal X$ be a smooth proper integral DM stack over $\mathbb C$ with projective coarse moduli space,  and let $ \mathbf \Delta \subseteq \mathcal X$ be a reduced divisor with at worst
normal crossing singularities. Suppose that some positive tensor power of  $\Omega^1_{\mathcal X}(\log \mathbf \Delta)$  contains a subsheaf
with big determinant. Then the first Chern class of every torsion-free coherent quotient sheaf of every positive tensor power of $\Omega^1_{\mathcal X}(\log \mathbf \Delta)$ is  pseudo-eﬀective.
\end{teoalpha}

\Cref{TA:SonCP-T4} is proved in \Cref{T:SonCP-T4}.  We refer the reader to \cite{CMZpositivity} where we discuss the notion of big and pseudo-effective line bundles on DM stacks.  We recall  that the case of \Cref{TA:SonCP-T4} where $\mathcal X=[V/G]$ for a smooth projective variety $V$ with a finite group $G$ acting on $V$ was proven in \cite[Thm.~7.6, Thm.~1.2]{CPFol19}, and in fact follows immediately also from the case of smooth projective varieties.  See also \cite[Thm.~0.1]{CPT11} for a related result in the non-uniruled case.    
Taking determinants in \Cref{TA:SonCP-T4}, one obtains the following corollary (proved in \Cref{T:SonCP-T1}), which is our primary interest in \Cref{TA:SonCP-T4}:

\begin{coralpha}

\label{TA:SonCP-T1}

Let $\mathcal X$ and $\mathbf \Delta$ be as in \Cref{TA:SonCP-T4}.
If some positive tensor power of $\Omega_{\mathcal X}^1(\log \mathbf \Delta)$ contains a subsheaf with big determinant, then $K_{\mathcal X}+\mathbf \Delta$ is big. \qed 
\end{coralpha}

One of the main applications of \Cref{TA:SonCP-T1} in the literature has been to the study of families of projective varieties $f:Y\to X$ (e.g., \cite{VZ02, KK08, Pat12, CPFol19, PS17, WW23}).  
The idea is that by considering variations of Hodge structures associated to  certain modifications of the family, one can sometimes construct Viehweg--Zuo sheaves, i.e., big subsheaves of tensor powers of $\Omega^1_X(\log \Delta)$, where $\Delta$ is the discriminant, which one assumes is a normal crossings  divisor.  
Popa--Schnell \cite{PS17}, building on a number of results in the literature, used the theory of Hodge modules to do this.  The upshot is that they were able to show that if $f:Y\to X$ is an algebraic fiber space of maximal variation between smooth projective varieties, $\Delta$ is any divisor containing the discriminant, and the geometric generic fiber of $f$ is of general type (or simply admits a good minimal model), then $K_X+\Delta$ is big.  Generalizing this result to Deligne--Mumford stacks is one of our primary motivations for proving \Cref{TA:SonCP-T4} and \Cref{TA:SonCP-T1}.  In forthcoming work, we will extend the techniques of Popa--Schnell to the case of the DM stacks considered here, which allows for the direct application of  \Cref{TA:SonCP-T4} and \Cref{TA:SonCP-T1}  to moduli stacks  that are not projective varieties, and that parameterize families of varieties with good minimal model, or families of KSBA stable pairs.
The present article is the third in a series with the aim of generalizing results of \cite{PS17, WW23}, on Viehweg hyperbolicity, to
the case of Deligne–Mumford stacks; the previous articles in this series are 
\cite{CMZpositivity, CMZslope_stability}.

\medskip 
Our proof of  \Cref{TA:SonCP-T4}  follows the proofs in \cite{CPFol19} and \cite{Schnell17Epi}.   The starting point is to proceed  by contradiction, which is to say, to assume that there is some torsion-free quotient $(\Omega^1_{\mathcal X}(\log \mathbf \Delta))^{\otimes N}\twoheadrightarrow \mathcal Q$ that is not pseudo-effective.  By definition, there would be a movable class $\alpha\in \operatorname{N}_1(\mathcal X)_{\mathbb R}$ 
 such that $c_1(\mathcal Q)\cdot \alpha<0$; we refer the reader to \cite{CMZslope_stability} where we discuss the notion of movable classes  on DM stacks. Dualizing the surjection $(\Omega^1_{\mathcal X}(\log \mathbf \Delta))^{\otimes N}\twoheadrightarrow \mathcal Q$, and using 
properties of $\alpha$-slope stability, one is able to construct a foliation $\mathcal F\subseteq \mathcal T_{\mathcal X}$ satisfying the condition that 
every nonzero quotient sheaf of $\mathcal F$ has positive $\alpha$-slope (see \Cref{S:SchnellProof}). We discuss foliations on DM stacks in  
\Cref{S:Foliations}; we refer the reader  to \cite{CMZslope_stability} where we discuss the notion of $\alpha$-slope stability on DM stacks, and where we prove the results used here regarding $\alpha$-slope stability for DM stacks.

Given this foliation, the next step is to establish that the foliation is algebraic (i.e., induced generically by the tangent spaces to the fibers of a rational map to a smooth projective variety).  In the special case of stacks $\mathcal X=[V/G]$, this was proven by Campana--P\u aun in  \cite[Thm.~1.1]{CPFol19}.  In this paper, we extend this result to the more general class of DM stacks considered here:

\begin{teoalpha}

\label{TA:SonCP-T8}

 Let $\mathcal X$ be a smooth proper integral DM stack over $\mathbb C$ with projective coarse moduli space, and let $\mathcal F\subseteq {\mathcal T}_{\mathcal X}$ be a foliation. 
 Suppose that there exists a movable class $\alpha\in \operatorname{N}_1(\mathcal X)_{\mathbb R}$ 
 such that every nonzero quotient
sheaf of $\mathcal F$ has positive $\alpha$-slope. Then $\mathcal F$ is an algebraic foliation.
 \end{teoalpha}
 
 We prove \Cref{TA:SonCP-T8} in \Cref{T:SonCP-T8}.
We recall that the case where $\mathcal X=[V/G]$ for a smooth projective variety $V$ with a finite group $G$ acting on $V$ was proven in \cite[Thm.~1.1]{CPFol19}, and in fact follows  also from the case of smooth projective varieties.  
In the case where $\mathcal X$ is a smooth projective variety and the class $\alpha$ is the class of a curve obtained as the  complete intersection of  
(very) ample hypersurfaces, the result is due to 
Bogomolov--McQuillan \cite{BMfoliations16}.

 Our strategy of proof for \Cref{T:SonCP-T8} is to use the approach of Campana--P\u aun for \cite[Thm.~1.1]{CPFol19}, and can be summarized as follows.  One considers the family of analytic leaves of the foliation $\mathcal F$, viewed as a family of analytic stacks with total space contained in $\mathcal X^{\operatorname{an}}\times \mathcal X^{\operatorname{an}}$ (parameterized by the first projection), and then one takes the Zariski closure.  In general, one does not expect to have much control over this Zariski closure, but under the hypotheses on the $\alpha$-slope of $\mathcal F$ and its quotients, together with results about $\alpha$-slope and vanishing of global sections that we establish in \cite{CMZslope_stability}, extending result in the literature for projective varieties, one can make an asymptotic Riemann--Roch calculation to show that the dimension of the Zariski closure in question is minimal (i.e., dimension of $\mathcal X$ plus the rank of the foliation).  From this one obtains that the Zariski closure of the general analytic leaf is  algebraic of the same dimension, and one argues that one obtains a rational map $\mathcal X\dashrightarrow Z$ to the resolution of singularities of an appropriate  parameter space for the Zariski closure of the fibers of the foliation, exhibiting the fibers, generically,  as the leaves of the foliation.

  \medskip 
Returning to the proof of \Cref{TA:SonCP-T4}, once one has the rational map $p:\mathcal X\dashrightarrow Z$ to a smooth projective variety $Z$ whose fibers generically induce the foliation $\mathcal F$, in \Cref{S:SchnellProof}, we turn to the strategy of proof in \cite{Schnell17Epi}. 
Namely, one can approach the proof via an inductive argument on the dimension, by reducing to the fibers of $p$ (in \S \ref{S:SchnellProof} we carefully organize the induction and proof by contradiction, and are simply streamlining the discussion here for brevity).   The strategy is to consider the divisor $K_{\mathcal X/Z}+ \mathbf \Delta^{hor}+\mathbf R(p)$, where $\mathbf \Delta^{hor}$ is the horizontal part of $\mathbf \Delta$ with respect to $p$, and $\mathbf R(p)$ is the ramification divisor for $p$ (see \Cref{S:WPLRCB} where the terminology is reviewed).  On the one hand, one shows from a direct computation that, given the set-up, one has the inequality $(K_{\mathcal X/Z}+ \mathbf \Delta^{hor}+\mathbf R(p))\cdot \alpha<0$, so that $K_{\mathcal X/Z}+ \mathbf \Delta^{hor}+\mathbf R(p)$ is not pseudo-effective.
On the other hand, inductively,  considering the fibers, one can deduce
 that $K_{\mathcal X/Z}+ \mathbf \Delta^{hor}+\mathbf R(p)$ is pseudo-effective, giving the contradiction.
 
 \smallskip 
The latter argument, reducing to the fibers to obtain that $K_{\mathcal X/Z}+ \mathbf \Delta^{hor}+\mathbf R(p)$ is pseudo-effective, relies, in the case of  smooth projective varieties, on results of Claudon--Kebekus--Taji \cite{CKT21}, which extend results of  Lu \cite{lu2002refinedkodairadimensioncanonical}, Campana \cite{Campana2004Orbifolds}, Berndtsson--P\u aun \cite{BP08bergman},  H\"oring \cite{horing10positivity},  Campana--P\u aun \cite{CP15}, 
Fujino \cite{Fujino17}, and 
P\u aun--Takayama \cite{PT18positivity}.  
We generalize  a special case of these results to the case of DM stacks in \Cref{T:CKT21-7.3} and 
\Cref{T:CKT21-7.1}; as the results in the literature are made for  singular varieties, unlike most of the other arguments in this paper, we prove these results by reducing to statements on the coarse moduli spaces.

\subsection*{Acknowledgements}
The first named author thanks Mihnea Popa for  conversations on the topic, which led to this project.  He also thanks Jonathan Wise and David Rydh for helpful conversations about the geometry of stacks. The second named author thanks his advisor, Mircea Musta\c{t}\u{a}, for useful discussions and all the support provided.
 The authors are also grateful to the organizers of the Simons Collaborations on Moduli of Varieties Workshop at the University of Utah in November 2024, where their  work on this project began.

\section{Preliminaries}

\subsection{Terminology}

We work over the complex numbers $\mathbb C$.  
A \emph{variety} is an integral separated scheme of finite type over $\mathbb C$. 
An \emph{alteration} $X'\to X$ is a surjective projective generically
finite   morphism of schemes over $\mathbb C$.  
We use the definition of a \emph{Deligne--Mumford (DM) stack} in \cite[Def.~4.1]{LMB}. Note that this differs from the definition in \cite{stacks-project} in that there is the additional hypothesis in \cite[Def.~4.1]{LMB} that the diagonal be representable, separated, and quasi-compact.  We direct the reader to \cite[App.~B]{CMW18} for a discussion of the relationship among various definitions of DM stacks in the literature (see in particular \cite[Fig.~1]{CMW18}).  We emphasize that, with the definition of DM stack that we are using, a morphism from a scheme to a DM stack is schematic (representable by schemes); see e.g., \cite[Lem.~B.20 and Lem.~B.12]{CMW18}.

    The general set-up in this paper will be a smooth proper (resp.~separated) integral DM stack $\mathcal X$ of finite type over $\mathbb C$ with coarse moduli space $\pi: \mathcal X\to X$, with the added assumption that the algebraic space $X$ be a projective (resp.~quasi-projective) variety. 
We refer the reader to \cite{CMZslope_stability, CMZpositivity}, as well, where we discuss certain properties of such DM stacks.

\subsection{Resolution of rational maps}  
Here we provide a  variation on \cite[Prop.~4.3]{KTbir23} that we will use later:

\begin{pro}[{\cite[Prop.~4.3]{KTbir23}}]\label{P:KT23-4.4}
Let $f:\mathcal X\dashrightarrow \mathcal Y$ be a rational map of smooth integral separated DM stacks of finite type over $\mathbb C$.  There is $2$-commutative diagram of smooth integral separated DM stacks of finite type over $\mathbb C$: 
$$
\xymatrix{
&\widetilde {\mathcal X} \ar[rd]^{\tilde f} \ar[ld]_\mu&\\
\mathcal X\ar@{-->}[rr]^f&&\mathcal Y
}
$$
such that $\mu$ is birational, and can be chosen to be an isomorphism over any open substack $\mathcal U\subseteq \mathcal X$ over which $f$ is defined.  If $\mathcal X$ and $\mathcal Y$ are proper over $\mathbb C$, then $\widetilde {\mathcal X}$ is proper over $\mathbb C$, and both $\mu$ and $f$ are proper.  If $\mathcal X$ and $\mathcal Y$ have (quasi-)projective coarse moduli spaces, then so does $\widetilde {\mathcal X}$. 
\end{pro}

\begin{proof} 
As in the proof of \cite[Prop.~4.3]{KTbir23}, let $\widetilde {\mathcal X}$ be a resolution of singularities 
(e.g., \cite[Thm.~1.3]{CMZslope_stability}) of $\Gamma(\mathcal X \dashrightarrow \mathcal Y)$, the closure of the graph of $f$ in $\mathcal X\times_{\mathbb C}\mathcal Y$.  One obtains the commutative diagram in the proposition from   \cite[Prop.~4.3(ii)]{KTbir23}.    

If $\mathcal X$ and $\mathcal Y$ are proper over $\mathbb C$, then clearly $\mathcal X\times_{\mathbb C}\mathcal Y$ and  $\Gamma(\mathcal X \dashrightarrow \mathcal Y)$ are proper over $\mathbb C$, and then so is $\widetilde {\mathcal X}$ 
(see \cite[Thm.~1.3]{CMZslope_stability}).  Any morphism of proper stacks over $\mathbb C$ is proper (\cite[\href{https://stacks.math.columbia.edu/tag/0CPT}{Lem.~0CPT}]{stacks-project}), giving that $\mu$ and and $\tilde f$ are proper.

Finally, if $\mathcal X$ and $\mathcal Y$ have (quasi-)projective coarse moduli spaces, then so does $\mathcal X\times_{\mathbb C}\mathcal Y$, and consequently, we have the same for $\Gamma(\mathcal X\dashrightarrow \mathcal Y)$. 
Here we are using that if $\mathcal Z\subseteq \mathcal W$ is a closed embedding into an integral separated stack $\mathcal W$ of finite type over $\mathbb C$ with quasi-projective coarse moduli space $W$, then the coarse moduli space of $\mathcal Z$ is quasi-projective, as it admits a finite map to $W$ 
(e.g.,  \cite[Lem.~A.1]{CMZslope_stability}).
 The fact that $\widetilde {\mathcal X}$ has (quasi-)projective coarse moduli space then follows from  \cite[Thm.~1.3]{CMZslope_stability}. 
\end{proof}

\begin{rem} Note that in \Cref{P:KT23-4.4}, even if $\mathcal Y$ is proper with projective coarse moduli space, it is not necessarily the case that 
the rational map $f$ can be extended to a morphism over an open substack  $\mathcal U\subseteq \mathcal X$ with complement of codimension at least $2$.  Indeed, consider the case where $X$ is a smooth projective variety, $\mathcal Y\to X$ is a non-trivial root stack construction over a smooth divisor on $X$, and $f:X\dashrightarrow \mathcal Y$ is the rational inverse.  
\end{rem}

\begin{rem}
Recall  \cite[Prop.~4.2]{KTbir23} that if $g : \mathcal X'\to \mathcal X$  is  a \emph{representable} proper birational
morphism of DM stacks of finite type over $\mathbb C$, with $\mathcal X'$ reduced and $\mathcal X$ normal, then the maximal open substack $\mathcal U \subseteq \mathcal  X$, over which $g$ restricts to an isomorphism $g^{-1}(\mathcal U)\to \mathcal U$ 
has complement everywhere of codimension $\ge  2$ in $\mathcal X$.  With this observation, the example of the root stack $f:X\dashrightarrow \mathcal Y$ in the previous remark shows that one cannot expect $\mu$ and  $\tilde f$ to be representable in \Cref{P:KT23-4.4}, even if $\mathcal Y$ is proper with projective coarse moduli space.

\end{rem}

\begin{rem}\label{R:KT23-4.4-Yprop}
If in \Cref{P:KT23-4.4} one assumes that $\mathcal Y=Y$ is a proper variety, 
then  there is an open substack $\mathcal U \subseteq \mathcal  X$ with complement of codimension $\ge 2$ such that $f$ restricts to a morphism on $\mathcal U$.  Indeed, let $p:U\to \mathcal X$ be an \'etale presentation.  The composition $U\to \mathcal X \dashrightarrow Y$ gives a rational map $f':U\dashrightarrow Y$.  Let $U'\subseteq U$ be an open subset such that there is a morphism $f':U'\to Y$ inducing $f':U\dashrightarrow Y$.  Since $Y$ is a proper variety, there is an open subset $U'\subseteq U''\subseteq U$ with the complement of $U''$ in $U$ of codimension $\ge 2$ such that we have a morphism $f':U''\to Y$ inducing  $f':U\dashrightarrow Y$.  As morphisms extend uniquely, the morphism $f': U''\to Y$ descends to a morphism $f:\mathcal U\to Y$, where $\mathcal U=p(U'')\subseteq \mathcal X$, inducing the rational map $f:\mathcal X\dashrightarrow Y$.   
\end{rem}

\section{Algebraicity of foliations}\label{S:Foliations}

 The proof of \cite[Thm.~7.6]{CPFol19} given in \cite[Thm.~1]{Schnell17Epi}  depends on \cite[Thm.~1.1]{CPFol19}. 
 In order to prove \Cref{TA:SonCP-T4}, generalizing \cite[Thm.~7.6]{CPFol19}, 
 in this subsection, we generalize \cite[Thm.~1.1]{CPFol19} to DM stacks.  
Recall that a foliation ${F}$ on a smooth variety $X$ is a coherent subsheaf $F\subseteq T_X$ of the tangent sheaf with the following properties:  (i) $F$ is closed under the Lie bracket, and
(ii) the quotient $T_X/F$ is torsion free.  Note that on a smooth DM stack $\mathcal X$, there is a well-defined Lie bracket $[-,-]$ on the tangent sheaf $\mathcal T_{\mathcal X}$.   Indeed, one can see this on an \'etale cover of $\mathcal X$ by smooth schemes, as \'etale covers give identifications of tangent spaces, and so  $[- , - ]$ descends to a morphism on the stack.   Consequently, we use the same definition of a foliation on a smooth DM stack:

\begin{dfn}[Foliation]
    A \emph{foliation} $\m{F}$ on a smooth DM stack $\X$ over $\mathbb C$ is a coherent subsheaf $\m{F}\subseteq T_{\X}$ with the following properties: 
    \begin{enumerate}[label=(\roman*)]
 \item $\m{F}$ is closed under the Lie bracket, and

   \item  The quotient $\mathcal T_{\X}/\m{F}$ is torsion free; i.e., $\mathcal F$ is saturated in $\mathcal T_{\mathcal X}$.
\end{enumerate}
\end{dfn}

\begin{rem}[Leaves of a foliation]\label{R:Leaves} 
Associated to $\mathcal X$ is a complex analytic stack $\mathcal X^{\operatorname{an}}$, with coarse moduli space $X^{\operatorname{an}}$, the analytic space associated to the coarse moduli space $X$ of $\mathcal X$.  A foliation $\mathcal F$ on $\mathcal X$ induces a foliation $\mathcal F^{\operatorname{an}}\subseteq \mathcal T_{\mathcal X^{\operatorname{an}}}$ on $\mathcal X^{\operatorname{an}}$; i.e., a saturated coherent subsheaf of the tangent sheaf, closed under bracket.  From the holomorphic version of the Frobenius theorem, if $\mathcal U\subseteq \mathcal X$ is the locus where $\mathcal F$ is locally free,  and we denote by $r$ the rank of $\mathcal F$, then  there is a decomposition of $\mathcal U^{\operatorname{an}}$ into a union of disjoint smooth connected analytic substacks $\{L_\alpha \}_{\alpha \in A}$ of $\mathcal U^{\operatorname{an}}$, called \emph{the leaves of the foliation}, with the following property: for every $\alpha$ we have $\mathcal F^{\operatorname{an}}|_{L_\alpha}\cong \mathcal T_{L_\alpha}$, and every point $x$ in  $\mathcal U^{\operatorname{an}}$ has an \'etale neighborhood $$p_x:\Omega_x\to \mathcal U^{\operatorname{an}}$$ and a 
holomorphic submersion $$\rho_x:\Omega_x\to \mathbb C^{n-r},$$ with connected fibers, such that for each $y\in \Omega_x$, the pre-image $p_x^{-1}(L_{p(y)}\cap p_x(\Omega_x))$ of the  intersection  of the leaf $L_{p_x(y)}$ of $\mathcal F^{\operatorname{an}}$ passing through $p_x(y)$ with $p_x(\Omega_x)$, is given by the fibers of $\rho_x$ containing the points in the preimage $p_x^{-1}(p_x(y))$.
 In other words, the leaves integrate $\mathcal F^{\operatorname{an}}|_{\mathcal U^{\operatorname{an}}}$, viewed as a distribution, and the leaves are locally the fibers of a holomorphic submersion.  Note that if we want, we are free to take $\Omega_x$ to be an analytic open inside of the analytification of an \'etale morphism $U_x\to \mathcal X$.
\end{rem}

Following \cite[Def.~4.1]{CPFol19}, and adapting it to the case of DM stacks, we make the following definition:

\begin{dfn}[Algebraic foliation]\label{D:AlgFol}
A foliation $\mathcal F\subseteq \mathcal T_{\mathcal X}$ on a smooth proper integral DM stack $\mathcal X$ over $\mathbb C$ with projective coarse moduli space is \emph{algebraic} if there exists a dominant rational map
$$
p:\mathcal X\dashrightarrow Z
$$
to a smooth  projective variety $Z$, such that there exists a dense  open substack $\mathcal U\subseteq \mathcal X$ over which $p$ is defined, so that denoting by $p:\mathcal U\to Z$ the induced morphism, we have 
\begin{equation}\label{E:AlgFolTXZ}
\mathcal F|_{\mathcal U} = \ker \left(Tp : {\mathcal T}_{\mathcal X}|_{\mathcal U}=\mathcal T_{\mathcal U} \to p^*{ T}_{Z}\right)=:\mathcal T_{\mathcal U/Z}.
\end{equation}
\end{dfn}

\begin{rem}\label{R:AlgFol}
Abusing notation, in the situation of \eqref{E:AlgFolTXZ}, we will often for convenience write that $$\mathcal F=\mathcal T_{\mathcal X/Z}$$ generically, even though $\mathcal T_{\mathcal X/Z}$ is not well-defined on $\mathcal X$.
\end{rem}

\begin{rem}\label{R:AlgFolRefin}
Replacing $Z$ in \Cref{D:AlgFol} by a smooth proper DM stack $\mathcal Z$ with projective coarse moduli space $Z$ gives an equivalent definition, as one can take a resolution of singularities of the coarse moduli space $Z$.
 We also observe that the leaves of an algebraic foliation are algebraic. 
 Indeed, let $\mathcal U'\subseteq \mathcal X$ be the locus where $p$ is defined, and, therefore where $\mathcal T_{\mathcal X/Z}$ is defined.  By \Cref{R:KT23-4.4-Yprop} we may assume that the complement of $\mathcal U'$ in $\mathcal X$ is of codimension $\ge 2$.  Note that $\mathcal T_{\mathcal X/Z}$ is saturated in $\mathcal T_{\mathcal X}|_{\mathcal U'}$,   as the quotient is a subsheaf of the locally free sheaf $p^*T_Z$, and is therefore torsion-free.  While we only required in \Cref{D:AlgFol} that $\mathcal F$ and $\mathcal T_{\mathcal X/Z}$ agree on an open substack $\mathcal U\subseteq \mathcal U'$, in fact,  we have $\mathcal F|_{\mathcal U'}=\mathcal T_{\mathcal U'/Z}$, as they are both saturated subsheaves of the vector bundle $\mathcal T_{\mathcal X}|_{\mathcal U'}$, which agree on a dense open subset (e.g., \Cref{R:ExtSatUnique}).  It follows from this that the leaves are algebraic, given by the fibers of $p$.  

\end{rem}

\begin{rem}\label{R:ExtSatUnique}
For lack of a better reference, we include the elementary argument that if $A$ is an integral domain with fraction field $K$, $M$ is an $A$-module, and $M',M''\subseteq M$ are saturated $A$-submodules such that $M'\otimes_A K = M''\otimes _AK\subseteq M\otimes_A K$, then $M'=M''$.  Indeed, we consider the short exact sequences
$$
\xymatrix{
0 \ar[r]& M' \ar[r]  \ar[d]& M\ar[r] \ar@{=}[d] & M/M' \ar[r] \ar[d]&0\\
0 \ar[r]& M'+M'' \ar[r] & M\ar[r] & M/(M'+M'') \ar[r]&0\\
}
$$  
By assumption we have that $(M'+M'')/M'$ is a torsion module (it is zero after tensoring $-\otimes_AK$).  On the other hand, by the Snake Lemma, it is isomorphic to a submodule of $M/M'$, which is torsion-free, by assumption.  So $(M'+M'')/M'=0$; i.e., $M'=M'+M''$.  But then we must have $M''\subseteq M'$.  Repeating the argument with $M''$ in the place of $M'$, we see that $M'=M''$.
\end{rem}

We now prove the following generalization of \cite[Thm.~1.1]{CPFol19}:

\begin{teo}

\label{T:SonCP-T8}

 Let $\mathcal X$ be a smooth proper integral DM stack over $\mathbb C$ with projective coarse moduli space, and let $\mathcal F\subseteq {\mathcal T}_{\mathcal X}$ be a foliation. 
 Suppose that there exists 
 a movable class $\alpha\in \operatorname{N}_1(\mathcal X)_{\mathbb R}$ 
 such that every nonzero quotient
sheaf of $\mathcal F$ has positive $\alpha$-slope. Then $\mathcal F$ is an algebraic foliation.
 \end{teo}

\begin{proof}
    We closely follow the proof by \cite{CPFol19} and make necessary adjustments along the way.   While the proof is fairly lengthy, the strategy of the proof is easy to describe:  One considers the analytic family $\mathbf \Lambda$ of leaves of the foliation, parameterized by the points of $\mathcal X^{\operatorname{an}}$ where the leaves are defined, as a family with total space contained in the product $\mathcal X^{\operatorname{an}}\times \mathcal X^{\operatorname{an}}$.  One then takes the Zariski closure of $\mathbf \Lambda$ in   $\mathcal X^{\operatorname{an}}\times \mathcal X^{\operatorname{an}}$ to obtain an algebraic substack $\mathcal W$ of $\mathcal X\times \mathcal X$.   The assumptions on the $\alpha$-slope of the foliation made in \Cref{T:SonCP-T8} imply that the dimension of $\mathcal W$ is equal to the dimension of $\mathbf \Lambda$ ($=\dim \mathcal X+\operatorname{rk}\mathcal F$); this is the heart of the matter, and is proved in \Cref{L:dimVdimLL} (which relies on \Cref{bound}).  One then shows that $\mathcal W$ is irreducible, and that the general fiber of the first projection $\mathcal W\to \mathcal X$ is also irreducible, so that the general fiber is the Zariski closure of an analytic leaf.  Finally, viewing the morphism $\mathcal W\to \mathcal X$ as a family of substacks of $\mathcal X$, one uses Chow varieties to obtain a rational map $\mathcal X\dashrightarrow Z$ to a smooth projective variety with general fiber given by leaves of the foliation, which are algebraic. 
    
  We now provide the details of the proof in four steps: 
    
 \subsubsection*{Step 1: Constructing the analytic family of leaves $\mathbf \Lambda$}
    Let $\mathbf E:=\operatorname{Sing}(\mathcal{F})\subseteq \mathcal{X}$, be the locus where $\mathcal F$ is not locally free, let $\mathcal{U}:=\mathcal{X}-\mathbf E$, and let $r:=\operatorname{rk} \mathcal F$.
    As in \Cref{R:Leaves}, every point $x$ in  $\mathcal U^{\operatorname{an}}$ has an \'etale neighborhood $p_x:\Omega_x\to \mathcal U^{\operatorname{an}}$ and a 
holomorphic submersion $\rho_x:\Omega_x\to \mathbb C^{n-r}$, with connected fibers, such that for each $y\in \Omega_x$, the pre-image of the  intersection $p_x^{-1}(L_{p(y)}\cap p_x(\Omega_x))$ of the leaf $L_{p_x(y)}$ of $\mathcal F^{\operatorname{an}}$ passing through $p_x(y)$ with $p_x(\Omega_x)$, is given by the fibers of $\rho_x$ containing the points in the preimage $p_x^{-1}(p_x(y))$. 
Denote by $\mathbf \Omega_x=p_x(\Omega_x)\subseteq \mathcal U^{\operatorname{an}}$ the image of $p_x$, which is open since $p_x$ is \'etale.    
     Let $\{\mathbf {\Omega}_i\}_{i\in I}$ be a countable locally finite subcover of $\mathcal{U}^{\operatorname{an}}$ from the open cover $\{\mathbf \Omega_x\}_{x\in \mathcal{U}^{\operatorname{an}}}$. 
     Let 
   \begin{align*}
  \mathbf {\Omega}&:=\bigcup_{i\in I} 
  (\mathbf {\Omega}_i \times \mathbf {\Omega}_i) \subseteq \mathcal{U}^{\operatorname{an}}\times \mathcal{U}^{\operatorname{an}}\\
  \mathbf \Lambda&:=\{(z,w)\in \mathbf  \Omega \, \mid \, \exists \, i\in I \, \text{ s.t. }  \, z\in \mathbf  \Omega_i, \,\, w \in L_z\cap \mathbf  \Omega_i\}\\
  & = \bigcup_{i\in I} \{(z,w)\in \mathbf \Omega_i\times \mathbf \Omega_i \mid w\in L_z\}\subseteq \mathbf {\Omega}\subseteq \mathcal U^{ \operatorname{an}}\times \mathcal U^{\operatorname{an}}
\end{align*}
    where $L_z$ is the leaf of $\mathcal{F}^{\operatorname{an}}$ containing $z$.  In other words, $\mathbf \Omega\subseteq \mathcal X^{\operatorname{an}}\times \mathcal X^{\operatorname{an}}$ is an open neighborhood of the open subset $\mathcal U^{\operatorname{an}}$ of the diagonal copy of $\mathcal X^{\operatorname{an}}$ inside of  $\mathcal X^{\operatorname{an}}\times \mathcal X^{\operatorname{an}}$, and $\mathbf \Lambda\subseteq \mathbf \Omega\subseteq \mathcal X^{\operatorname{an}}\times \mathcal X^{\operatorname{an}}$ is, via the first projection, the family of analytic leaves parameterized by the points of $\mathcal U^{\operatorname{an}}$.

\subsubsection*{Step 2: Constructing the algebraic family of leaves  $\mathcal W$}
 Considering $\mathbb C$-points under the continuous map of topological spaces  $a:|\mathcal{X}^{\operatorname{an}}\times \mathcal{X}^{\operatorname{an}}|\rightarrow |\mathcal{X}\times \mathcal{X}|$, let $\mathcal W$ be the reduced induced closed substack of $\mathcal X$ supported on the Zariski closure of $a(\mathbf \Lambda)$; i.e., 
$$
\mathcal W=\overline{\mathbf \Lambda}^{\operatorname{Zar}}\subseteq \mathcal X\times \mathcal X.
$$

\subsubsection*{Step 3: The geometry of $\mathcal W$}
   The main claim, which we prove below in \Cref{L:dimVdimLL}, is that the assumptions on the $\alpha$-slope of the foliation made in \Cref{T:SonCP-T8} imply that 
\begin{equation}\label{E:Clm-dimW}
  \dim(\mathcal W)=\dim(\mathbf {\Lambda})=\dim (\mathcal X)+r. 
\end{equation}

    Using \eqref{E:Clm-dimW} we can give a good description of the geometry of $\mathcal W$.    
     First, from \eqref{E:Clm-dimW}, we can conclude that $\mathcal W$ is irreducible.  Indeed,  $\mathcal W^{\operatorname{an}}$ contains the connected smooth (and therefore irreducible) analytic substack $\mathbf \Lambda$, and there can be only one irreducible component of $\mathcal W^{\operatorname{an}}$ containing a given irreducible analytic substack of the same dimension as $\mathcal W^{\operatorname{an}}$.
   Let $$\pi_{\mathcal W}:\mathcal W\longrightarrow  \mathcal X$$ be the restriction of the first projection $\mathcal X\times \mathcal X \to \mathcal X$ to $\mathcal W$.  
Looking at the dimensions, a general fiber of $\pi_{\mathcal W}$ is of dimension $r$.  As the leaves of the foliation are also of dimension $r$, the general fibers of $\pi_{\mathcal W}$ are closed algebraic subsets of dimension $r$ containing  analytic leaves of the foliation as open analytic subsets. In other words, the Zariski closure of these leaves are irreducible components of these fibers.  In fact, we claim that the general fiber of $\pi_{\mathcal W}$ is irreducible, and therefore equal to the Zariski closure of the general leaf of $\mathcal F$.

To show that  the general fiber of $\pi_{\mathcal W}$ is irreducible, it suffices to consider the normalization $\nu:\hat{\mathcal W}\to \mathcal W$ (see \cite[\href{https://stacks.math.columbia.edu/tag/0GMH}{\S 0GMH}]{stacks-project} for the definition of the normalization), and then show that the general fiber of the composition $\pi_{\hat{\mathcal W}}:= \pi_{\mathcal W}\circ \nu:\hat{\mathcal W}\to \mathcal X$ has irreducible general fiber.   
The first step is to show that the general fiber of $\pi_{\hat {\mathcal W}}$ is connected.  For this we first observe that $\mathbf \Lambda$ being smooth (and therefore normal), the inclusion $\mathbf \Lambda \to \mathcal W^{\operatorname{an}}$ lifts to a morphism $\mathbf \Lambda \to \hat{\mathcal W}^{\operatorname{an}}$ to the normalization, by the universal property of the normalization   (e.g., \cite[Ch.~8, \S 3, Prop., p.164]{GrRe84}).  Therefore, $\pi_{\hat{\mathcal W}}$ admits a section, given by the diagonal morphism $\Delta^{\operatorname{an}}:\mathcal X^{\operatorname{an}}\hookrightarrow \mathbf \Lambda$.  
Now, using the Stein Factorization of $\pi_{\hat{\mathcal W}}$ (see  \cite[Thm.~4.6.14]{AlperStacksBook}), and the existence of a section of $\pi_{\hat{\mathcal W}}$, one can immediately conclude from the irreducibility of $\hat {\mathcal W}$ that the general fiber of $\pi_{\hat {\mathcal W}}$ is connected.   Now consider the associated morphism $\pi_{\hat W}:\hat W\to X$ on the coarse moduli spaces.  This is a surjective morphism of normal projective varieties with connected general fiber.  It then follows  (e.g., \cite[Thm.~1.20]{fujita_L}, \cite[Thm.~1.20']{fujita_L_corr}, \cite[Prop.~5.51]{GW20})  that the general fiber of $\pi_{\hat W}$ is  integral.  This implies that the general fiber of $\pi_{\hat {\mathcal W}}$ is integral, establishing the claim.  (Note that one could have started this paragraph with the coarse moduli spaces, and made the argument entirely at the level of analytic spaces and varieties.)

\subsubsection*{Step 4: Morphism to the Chow variety}

We now wish to view the morphism $\pi_{\mathcal W}:\mathcal W\to \mathcal X$ as a family of substacks of $\mathcal X$, parameterized by $\mathcal X$, so that we can construct a rational map from $\mathcal X$ to an appropriate moduli space of substacks, with fibers that are generically the leaves of $\mathcal F$.  However, rather than considering Chow varieties or Hilbert schemes of stacks,  we find it convenient to modify the picture slightly, so that we can work with Chow varieties, or Hilbert schemes, of projective varieties.  

To this end, let $q:V\to \mathcal X$ be a finite flat morphism from a smooth projective variety $V$, and consider the diagram
$$
\xymatrix{
W' \ar@{^(->}[r] \ar[d]& \mathcal X\times V\ar[d]^{1\times q}\\
\mathcal W \ar@{^(->}[r] \ar[rd]_{\pi_{\mathcal W}}& \mathcal X\times \mathcal X \ar[d]^{pr_1}\\
& \mathcal X
}
$$
where we define $W'$ to be the fibered product.   The fiber of $W'\to \mathcal X$ over a general point $x$ is the pre-image under $q$ of the algebraic closure of the leaf of $\mathcal F$ through $x$.  In particular, along the algebraic closure of the leaf through a general point $x$ of $\mathcal X$, the fibers of $W'\to \mathcal X$ are constant.  Moreover, for general points of $\mathcal X$ that are in different leaves, the fibers of $W'\to \mathcal X$ over those points are distinct.   
 While one can establish that there is a non-empty open substack of $\mathcal X$ over which $W'\to \mathcal X$ is flat, so that one can consider using  Hilbert schemes, we prefer to observe instead that there is a non-empty open substack of $\mathcal X$ over which the fibers of $W'\to \mathcal X$ lie in the same component of the Chow variety of $V$ (take an \'etale cover of $\mathcal X$ and argue as in the case of varieties). 
 We therefore have a rational map 
$$
\mathcal X\dashrightarrow \operatorname{Chow}(V)
$$
Note, that as $ \operatorname{Chow}(V)$ is a scheme, this morphism factors rationally through the coarse moduli space, which is a normal variety. As such, the morphism is well defined away from a codimension 2 locus. Let $Z$ be the closure of the image, and take a resolution of singularities of $Z$, if $Z$ is not smooth.  This provides the rational map $p:\mathcal X\dashrightarrow Z$ with generic fibers being the leaves of $\mathcal F$.  This implies that $\mathcal T_{\mathcal X/Z}$ agrees with $\mathcal F$ generically. 
\end{proof}

    Now we have to show the main claim \eqref{E:Clm-dimW}:
    
\begin{lem}\label{L:dimVdimLL}
$  \dim(\mathcal W)=\dim(\mathbf {\Lambda})=\dim (\mathcal X)+r$, where $r=\operatorname{rk}(\mathcal F)$. 
\end{lem}

\begin{proof} We clearly have $ \dim  ( \mathcal X)+r=\dim(\mathbf {\Lambda})\le  \dim(\mathcal W) $, so we only have to show the inequality  $\dim (\mathcal W)\le \dim
( \mathcal X)+r$.  
  As $\mathcal W$ is a closed substack of $\mathcal X\times \mathcal X$, its coarse moduli space $W$ is a closed subscheme of the coarse moduli space $X\times X$ of $\mathcal X\times \mathcal X$, and it suffices to show that $\dim W\le \dim (\mathcal X)+r$.  
To this end, let 
 $\mathcal L$ be the pull-back to $\mathcal X\times \mathcal X$  of an ample line bundle $L$ on  $ {X}\times  {X}$, and note that  $\mathcal L|_{\mathcal W}$ is the pull-back of the ample line bundle $L|_W$ on $W$.   Setting $$n=\dim (\mathcal X),$$ it suffices to show that there is some positive constant $C$ such that for sufficiently large $k$, we have 
  \begin{equation}\label{1101}
 h^0(W,L|_W^{\otimes k})\le C\cdot k^{n+r}.
 \end{equation}
 Since pulling back sections of $L|_W$ to sections of $\mathcal L|_{\mathcal W}$ is injective, to establish \eqref{1101}, it suffices to show that there is a (possibly different) positive constant $C$ such that for all sufficiently large $k$ we have 
   \begin{equation}\label{1102}
 h^0(\mathcal W,\mathcal L|_{\mathcal W}^{\otimes k})\le C\cdot k^{n+r}.
 \end{equation}
 
 Let $\mathcal L^{\operatorname{an}}$ be the  analytic line bundle on $\mathcal X^{\operatorname{an}}\times \mathcal X^{\operatorname{an}}$ associated with $\mathcal L$.  
 For any $k\geq 0$, the sections of $\mathcal L|_{\mathcal W}^{\otimes k}$ on $\mathcal W$ restrict injectively to sections of $\mathcal L^{\operatorname{an}}|_{\mathbf \Lambda}^{\otimes k}$ over $\mathbf \Lambda$ ($\mathcal W$ is the Zariski closure of $\mathbf \Lambda$). 
   Next, one considers the further restriction of these sections to formal neighborhoods of the diagonal, along  $\mathcal U\subseteq \mathcal X$; we use the notation  $\mathcal{X}_0:= \Delta^{\operatorname{an}}(\mathcal{U}^{\operatorname{an}})\subseteq \mathbf \Lambda$, where $\Delta^{\operatorname{an}}$ is the diagonal map.  
For any $m>0$, let $\mathcal{X}_m$ be the $m^{\rm th}$ infinitesimal neighborhood of $\mathcal{X}_0$ in ${\mathbf \Lambda}$, defined by the structure sheaf:
$\OO_{\mathcal{X}_m}:=\OO_{\mathbf \Lambda}/\mathcal I_{0}^{m+1},$
where $\mathcal I_{0}$ is the sheaf of ideals of the diagonal $\mathcal{X}_0\subseteq \Lambda$.  Considering the inclusions $\mathcal X_m\hookrightarrow \mathbf \Lambda\hookrightarrow \mathcal W^{\operatorname{an}}$, we see that to establish \eqref{1102}, 
it is enough to show that  there exists a (possibly different) positive constant $C$ such that for all $m$ and all sufficiently large $k$ one has 
 \begin{equation}\label{1104}
h^0(\mathcal{X}_m, \mathcal L|_{\mathcal X_m}^{\otimes k}\otimes \OO_{\mathcal{X}_m})\leq C\cdot k^{n+r}.
\end{equation}

 We now want to relate this to global sections of the foliation $\mathcal F$.  To do this, we observe that the restriction of the foliation $\mathcal F$ to $\mathcal U$  is isomorphic to the normal bundle to the diagonal in $\Lambda$; i.e., 
   \begin{equation}\label{1103}
 \mathcal F^{\operatorname{an}}|_{\mathcal U^{\operatorname{an}}}\cong \mathcal N_{\mathcal X_0 |\mathbf \Lambda};
 \end{equation}
 the identification above follows immediately from the standard identification for foliations on varieties, after taking an \'etale cover.   In a little more detail, identifying $\mathcal X_0=\mathcal U^{\operatorname{an}}$ for brevity, one has the standard identification $\mathcal N_{\mathcal X_0|\mathcal X_0\times \mathcal X_0}\cong \mathcal T_{\mathcal X_0}$, of the normal bundle to the diagonal with the tangent bundle, as well as a natural inclusion of normal bundles $\mathcal N_{\mathcal X_0|\mathbf \Lambda}\hookrightarrow \mathcal N_{\mathcal X_0|\mathcal X_0\times\mathcal X_0}$. Together this gives an inclusion $\mathcal N|_{\mathcal X_0|\mathbf \Lambda} \hookrightarrow \mathcal T_{\mathcal X_0}$; after an \'etale cover, one is reduced to the case of varieties, where one can then confirm that the inclusion  $\mathcal N|_{\mathcal X_0|\mathbf \Lambda} \hookrightarrow \mathcal T_{\mathcal X_0}$ identifies $\mathcal N|_{\mathcal X_0|\mathbf \Lambda}$ with $\mathcal F^{\operatorname{an}}|_{\mathcal X_0}\subseteq \mathcal T_{\mathcal X_0}$.

To relate this to our previous estimate, \eqref{1104}, we use the standard short exact sequence of sheaves on $\mathcal X_0$, 
$$
0\to \operatorname{Sym}^m(\mathcal N^\vee _{\mathcal X_0|\mathbf \Lambda})\to \mathcal O_{\mathcal X_{m+1}}\to \mathcal O_{\mathcal X_m}\to 0.
$$
Finally, this together with \eqref{1103} shows that to establish \eqref{1104}, and therefore to complete the proof of \Cref{L:dimVdimLL}, it is sufficient to establish that there exists a (possibly different) positive constant $C$ such that, for any sufficiently large $k$:
\begin{equation}\label{1106}
\sum_{m\geq 0} h^0(\mathcal{X}_0, \mathcal L|_{\mathcal X_0}^{\otimes k}\otimes \operatorname{Sym}^m((\mathcal{F}^{\operatorname{an}}|_{\mathcal U})^\vee))\leq C\cdot k^{n+r}.
\end{equation}

 The estimate \eqref{1106} will be a consequence
of following statement:

\begin{lem}\label{bound}  Let $\mathcal{F}$ be a coherent sheaf on $\mathcal X$, which is locally free when restricted to an open set $\mathcal{U}\subseteq \mathcal{X}$ such that $\operatorname{codim}_{\mathcal{X}}(\mathcal{X}- \mathcal{U})\geq 2$, let 
$\alpha\in \operatorname{N}_1(\mathcal X)_{\mathbb R}$ 
 be a movable class,
and let $\mathcal M$ be the pull back to $\mathcal X$ of an ample line bundle on the coarse moduli space $X$.  
Let $\delta_0$ be a positive integer such that  
\begin{equation}\label{E:L:bound}
\displaystyle \delta_0>\frac{c_1(\mathcal M)\cdot \alpha}{\mu_{\alpha}^{\min}(\mathcal{F})}.
\end{equation}
 Let $q:V\rightarrow \mathcal{X}$ be a finite flat morphism from a smooth projective variety.  Under the assumption that $$\mu_{\alpha}^{\min}(\mathcal F)> 0,$$ the following assertions are true:

\begin{enumerate}[label=(\alph*)]

\item \label{bound:a}We have $H^0\left(\mathcal{U},(\mathcal M^{\otimes k}\otimes \operatorname{Sym}^m(\mathcal{F}^\vee))|_{\mathcal U}\right)=0$ if  $m\geq \delta_0 k$.

\item  \label{bound:b} There exists a smooth projective variety $Y$ of dimension
$\dim(Y)= \dim(\mathcal X)+ \operatorname{rk}(\mathcal{F})-1$, together with a map $p: Y\to V$ and a line bundle $B$ on $Y$ such that
\begin{equation}\label{anym}
p_*(B^{\otimes m})= ({\operatorname{Sym}^m}(q^*\mathcal{F}^\vee))^{\vee \vee}
\end{equation}  
for any $m\geq 1$.

\item  \label{bound:c} For any pair of positive integers $k$ and $m$,  we have the
  inequality
\begin{equation}\label{ineq}
  h^0\left(\mathcal{U}, (\mathcal M^{\otimes k}\otimes \operatorname{Sym}^m(\mathcal{F}^\vee))|_{\mathcal U}\right)\leq 
  h^0\left(Y, p^*q^*\mathcal M^{\otimes k}\otimes B^{\otimes m}\right).
\end{equation}    

\end{enumerate}

\end{lem}

 Before proving \Cref{bound}, we explain how it 
implies  the 
inequality \eqref{1106}, and therefore completes the proof of \Cref{L:dimVdimLL}. 
First note that we may take $\mathcal L=\mathcal M'\boxtimes \mathcal M'$ for some line bundle $\mathcal M'$ on $\mathcal X$ that is the pull-back of an ample line bundle from $X$.  Then, the left hand side of \eqref{1106} can be written as   $\sum_{m\geq 0} h^0(\mathcal{X}_0, \mathcal L|_{\mathcal X_0}^{\otimes k}\otimes \operatorname{Sym}^m((\mathcal{F}^{\operatorname{an}}|_{\mathcal U})^\vee))= \sum_{m\geq 0} h^0(\mathcal U, \mathcal M'|_{\mathcal U}^{\otimes 2k}\otimes \operatorname{Sym}^m((\mathcal{F}^{\operatorname{an}}|_{\mathcal U})^\vee))=  \sum_{m\geq 0} h^0\left(\mathcal{U},((\mathcal M'^{\otimes 2})^{\otimes k}\otimes \operatorname{Sym}^m(\mathcal{F}^\vee))|_{\mathcal U}\right)$.  Setting $\mathcal M=\mathcal M'^{\otimes 2}$, we see that to show \eqref{1106}, it suffices to show that 
there exists a (possibly different) positive constant $C$ such that, for any sufficiently large $k$:
\begin{equation}\label{ex1-0}
  \sum_{m\geq 0} h^0\left(\mathcal{U},(\mathcal M^{\otimes k}\otimes \operatorname{Sym}^m(\mathcal{F}^\vee))|_{\mathcal U}\right) \le C \cdot k^{n+r}.
  \end{equation}
For this we observe that from \Cref{bound}\ref{bound:a} we have that 
\begin{equation}\label{ex1}
  \sum_{m\geq 0} h^0\left(\mathcal{U},(\mathcal M^{\otimes k}\otimes \operatorname{Sym}^m(\mathcal{F}^\vee))|_{\mathcal U}\right) =
  \sum_{0\le m\leq \delta_0k}h^0\left(\mathcal{U},(\mathcal M^{\otimes k}\otimes \operatorname{Sym}^m(\mathcal{F}^\vee))|_{\mathcal U}\right);
\end{equation}  
it turns out that this (in particular \Cref{bound}\ref{bound:a}) is the only place in the agument where we use the assumption that  $\mu_{\alpha}^{\min}(\mathcal{F})>0$.   Consequently, using also \Cref{bound}\ref{bound:c}, we see that to establish \eqref{ex1}, 
 it suffices to show that 
there exists a (possibly different) positive constant $C$ such that, for any sufficiently large $k$:
\begin{equation}\label{1107}
 \sum_{0\le m\leq \delta_0k}  h^0\left(Y, p^*q^*\mathcal M^{\otimes k}\otimes B^{\otimes m}\right) \le C \cdot k^{n+r}.
\end{equation}
Now letting $H$ be a very ample line bundle on $Y$ such that $H\otimes B^{-1}$ is effective, we have an inclusion of line bundles, $B\to H$, so that applying global sections we have that 
$$
h^0(Y,p^*q^*\mathcal M^{\otimes k}\otimes B^{\otimes m})\le h^0\big(Y, p^*q^*\mathcal M^{\otimes k}\otimes H^{\otimes m}\big).
$$
To show that  the left  hand side of  
\eqref{1107} is $O(k^{n+r})$,
it therefore suffices to show that 
\begin{equation}\label{1ex1}
\sum_{0\le m\leq \delta_0k} h^0\left(Y, p^*q^*\mathcal M^{\otimes k}\otimes H^{\otimes m}\right)
\end{equation}  
is $O(k^{n+r})$.  But the above is obviously smaller than $\sum_{0\le m\leq \delta_0k} h^0\left(Y, p^*q^*\mathcal M^{\otimes k}\otimes H^{\otimes \delta_0k}\right)=  \delta_0k\cdot h^0\left(Y, p^*q^*\mathcal M^{\otimes k}\otimes H^{\otimes \delta_0k}\right)$.  By the asymptotic Riemann--Roch theorem, we have the estimate that 
$h^0\left(Y, p^*q^*\mathcal M^{\otimes k}\otimes (H^{\otimes \delta_0})^{\otimes k}\right) $ grows like  $O(k^{\dim Y})= O(k^{n+r-1})$ as $k\to \infty$ (take $H$ sufficiently ample that $p^*q^*\mathcal M\otimes H^{\delta_0}$  is ample).  
This completes the explanation of how \Cref{bound}
implies  the 
inequality \eqref{1106}, and therefore completes the proof of \Cref{L:dimVdimLL}.
We now proceed with the proof of \Cref{bound}:

\begin{proof}[Proof of \Cref{bound}]
Part \ref{bound:a} follows from  
\cite[Lem.~3.4(b)]{CMZslope_stability}
and the following slope inequality, which holds if $m\geq \delta_0k$: 
\begin{align*}
\mu_{\alpha}^{\max} \big(\mathcal M^{\otimes k}\otimes (\operatorname{Sym}^m(\mathcal{F}^\vee))^{\vee\vee}\big)&=k\cdot \mu_{\alpha}^{\max} \big(\mathcal M)+\mu_\alpha^{\max} (\operatorname{Sym}^m(\mathcal{F}^\vee))^{\vee\vee}\big) & 
\text{(\cite[Thm.~3.23(a)]{CMZslope_stability})} \\
&=k \cdot c_1(\mathcal M)\cdot\alpha+m\cdot \mu_{\alpha}^{\max}(\mathcal{F}^\vee) & 
\text{(\cite[Thm.~3.27(a)]{CMZslope_stability})} \\
&=k \cdot c_1(\mathcal M)\cdot\alpha-m\cdot \mu_{\alpha}^{\min}(\mathcal{F})& 
\text{(\cite[Lem.~3.10(1)]{CMZslope_stability})} \\
&<0 & \text{(From \eqref{E:L:bound})}
\end{align*}

Part \ref{bound:b} is described by Nakayama
\cite[Ch.~V, \S 3.22 and Ch.~III Prop.~5.10]{nakayama_2004}.  There he shows a more general result that implies this. We will simply recall the construction of
$(Y, B)$ for the convenience of the reader. Let $V_0:=q^{-1}(\mathcal{U})$.
Let $\pi:\mathbb{P}(q^*\mathcal{F}^\vee)\to V$ be the
scheme over $V$ associated to the symmetric algebra on the torsion-free coherent sheaf $q^*\mathcal{F}^\vee$,
and let $\OO_{q^*\mathcal{F}^\vee}(1)$ be the tautological line bundle on $\mathbb{P}(q^*\mathcal{F}^\vee)$.
Let $\mathbb{P}^\prime(q^*\mathcal{F}^\vee)$ be the normalization of the component of
$\mathbb{P}(q^*\mathcal{F}^\vee)$ that contains the Zariski open subset $\pi^{-1}(V_0)$ (we recall the crucial fact that the co-dimension of $V_0$ in $V$ is at least  two).
Finally, let $Y$ be a smooth projective variety such that there exists a birational morphism $Y\to \mathbb{P}^\prime(q^*\mathcal{F}^\vee)$, which is an isomorphism over $\pi^{-1}(V_0)$,  and such that the complement is snc.
We denote by $\mu: Y\to \mathbb{P}(q^*\mathcal{F}^\vee)$ the resulting map, and let
\begin{equation}\label{ex2}
p:Y\to V
\end{equation}  
be the composition $\pi\circ \mu$. Nakayama shows that we can take
\begin{equation}\label{ex3}
  B:= \mu^*\left(\OO_{q^*\mathcal{F}^\vee}(1)\right)+ D_1,
\end{equation}
where $D_1$ is an
effective $p$-exceptional divisor. The important fact here is that $B$ can be chosen so that \eqref{anym} holds \emph{for any} $m$.

Part \ref{bound:c} follows from part \ref{bound:b} and the projection formula, using that reflexive sheaves are $S_2$. 
We have
\begin{align*}
  h^0\left(\mathcal{U}, (\mathcal M^{\otimes k}\otimes \operatorname{Sym}^m(\mathcal{F}^\vee))|_{\mathcal U}\right) & \le  h^0\left(V_0, q^*((\mathcal M^{\otimes k}\otimes \operatorname{Sym}^m(\mathcal{F}^\vee))|_{\mathcal U})\right)\\
  &= h^0(V,q^*\mathcal M^{\otimes k}\otimes (\operatorname{Sym}^m(q^*\mathcal F^\vee))^{\vee \vee})\\
  &= h^0(V,q^*\mathcal M^{\otimes k}\otimes p_*B^{\otimes m})\\
  &= h^0(Y,p^*q^*\mathcal M^{\otimes k} \otimes B^{\otimes m}).
  \end{align*}
This completes the proof of \Cref{bound}.
\end{proof}

Having completed the proof of \Cref{bound}, we have completed the proof of \Cref{L:dimVdimLL}.  
\end{proof}

\section{Weak positivity for log relative canonical bundles}\label{S:WPLRCB}
 The proof of \cite[Thm.~7.6]{CPFol19} given in \cite[Thm.~1]{Schnell17Epi},  depends on \cite[Thm.~7.1 and 7.3]{CKT21} and  \cite[Thm.~3.4]{CPFol19}. 
  In order to prove \Cref{TA:SonCP-T4}, generalizing \cite[Thm.~7.6]{CPFol19}, 
 in this subsection, we generalize \cite[Thm.~7.1 and 7.3]{CKT21} and  \cite[Thm.~3.4]{CPFol19}  to DM stacks.

The set-up will always be the following.  We will take $\mathcal X$ to be a smooth proper integral  DM stack of finite type over $\mathbb C$ with projective coarse moduli space, $p:\mathcal X\to Z$ will be an equidimensional morphism with connected fibers to a smooth projective variety $Z$, and we will assume that the general fiber is smooth. 
We denote by $K_{\mathcal X/Z}$ the relative dualizing sheaf,  
which since both $\mathcal X$ and $Z$ are smooth, is a line bundle, linearly equivalent to $K_{\mathcal X}-p^*K_Z$.
By the definition of a coarse moduli space, there is a unique morphism $f:X\to Z$ making the following diagram commute:
\begin{equation}\label{E:CKT-cm}
\xymatrix{
\mathcal X  \ar[d]_\pi \ar[r]^p & Z \\
X \ar[ru]_f
}
\end{equation}

\begin{rem}[Log canonical singularities]\label{R:lc-pair}
Let $\mathcal X$ be a smooth separated integral DM stack of finite type over $\mathbb C$ with quasi-projective coarse moduli space $\pi:\mathcal X\to X$,  let  $\mathbf \Delta\subseteq \mathcal X$ be a reduced divisor with at worst normal crossing singularities, and let $\mathbf R$ be the ramification divisor for $\pi$. If $\Delta$ (resp.~$R$) is the $\mathbb Q$-divisor on $X$ such that $\pi^*\Delta=\mathbf \Delta$ (resp.~$\pi^*R=\mathbf R$), then the pair $(X,R+\Delta)$ is log canonical. 
Indeed, if $D\subseteq |\Delta|\cap |R|$ is an irreducible divisor with non-trivial automorphism group of order $r$ (the length of the generic stabilizer of the divisor minus the length of the generic stabilizer of the stack) at the generic point of $\pi^*D$, then the coefficient of $D$ in $\Delta$ is $\frac{1}{r}$ and in $R$ is $(1-\frac{1}{r})$.  
Therefore, the coefficient of $D$ in $R+\Delta$ is $1$.  
Therefore,  $(X,R+\Delta)$
 will be log canonical, since the pair $(X,R+\Delta)$ pulls back via $\pi$ to the snc pair $(\mathcal X,\mathbf \Delta)$; see e.g., \cite[Cor.~2.43(2)]{kollar_singularities_MMP}.
\end{rem}

We start with the following generalization of a special case of \cite[Thm.~7.3]{CKT21}, which in turn extends results of 
 \cite{lu2002refinedkodairadimensioncanonical, Campana2004Orbifolds, BP08bergman, horing10positivity, Fujino17, PT18positivity}: 

\begin{teo}
\label{T:CKT21-7.3}
Let $\mathcal X$ be a smooth proper integral DM stack of finite type over $\mathbb C$ with projective coarse moduli space, and let $p:\mathcal X\to Z$ be a surjective morphism with connected fibers to a smooth projective variety $Z$, with smooth general fiber. 
Let $\mathbf \Delta\subseteq \mathcal X$ be a reduced divisor with at worst normal crossing singularities. 
If the restriction of $K_{\mathcal X/Z} + \mathbf \Delta$ to the general fiber of $p$ is pseudo-effective, then $K_{\mathcal X/Z} +\mathbf \Delta$ is
pseudo-effective.
\end{teo}

We are using the definition of pseudo-effective line bundles on DM stacks from \cite[Def.~2.4]{CMZpositivity}; i.e., $\mathcal L$ is pseudo-effective on $\mathcal X$ if some positive multiple descends to a pseudo-effective  line bundle on the coarse moduli space $X$.

\begin{proof} Using the notation in the diagram \eqref{E:CKT-cm},  let $\Delta$ (resp.~$R$) be the $\mathbb Q$-divisor on $X$ such that $\pi^*\Delta=\mathbf \Delta$ (resp.~$\pi^*R=\mathbf R$), where $\mathbf R$ is the ramification divisor for $\pi$.    
Starting with the fact that  $K_{\mathcal X}=\pi^*K_X+\mathbf R$, we have that $K_{\mathcal X/Z}=K_{\mathcal X}-p^*K_Z=(\pi^*K_X-\pi^*f^*K_Z)+\mathbf R$, so that $K_{\mathcal X/Z}=\pi^*(K_{X/Z}+R)$; here we are using that $K_{X/Z}=K_X-f^*K_Z$, e.g., \cite[Rem.~26(vii)]{kleiman80rel-duality}.
Consequently, $K_{\mathcal X/Z}+\mathbf \Delta= \pi^*(K_{X/Z}+R+\Delta)$, 
so it suffices to show that $K_{X/Z}+R+\Delta$ is pseudo-effective on $X$. 

 If $\mathcal F$ is a general fiber of $p$, then we have
\begin{equation}\label{E:CKT21-7.3E1}
(K_{\mathcal X/Z}+\mathbf \Delta)|_{\mathcal F}=K_{\mathcal F}+\mathbf \Delta|_{\mathcal F},
\end{equation}
which we are assuming  is pseudo-effective.  
Denote by $F$ the fiber of $f$ corresponding to the fiber $\mathcal F$ of $p$, and note that $\pi|_{\mathcal F}:\mathcal F\to F$ exhibits $F$ as the coarse moduli space of $\mathcal F$.  
Since we are considering a general fiber, we have that $(\pi|_{\mathcal F})^*(K_F+R|_F)= K_{\mathcal F}$, 
 so that 
\begin{equation}\label{E:CKT21-7.3E2}
 (\pi|_{\mathcal F})^*(K_F+R|_F+\Delta|_F)=K_{\mathcal F}+\mathbf \Delta|_{\mathcal F}.
\end{equation}  
Since we are assuming that $(K_{\mathcal X/Z}+\mathbf \Delta)|_{\mathcal F}$ is pseudo-effective,  it follows from \eqref{E:CKT21-7.3E1} and \eqref{E:CKT21-7.3E2} and the definition of a line bundle being pseudo-effective on $\mathcal F$ that $K_F+R|_F+\Delta|_F$ is pseudo-effective on $F$.

It then follows from \cite[Thm.~7.3]{CKT21} that $K_{X/Z}+R+\Delta$ is pseudo-effective.   Note that the statement of \cite[Thm.~7.3]{CKT21} is for pairs consisting of a smooth projective variety and an snc boundary, but the proof of \cite[Thm.~7.3]{CKT21}, which relies on \cite[Thm.~1.1]{Fujino17}, holds under the more general hypotheses of \cite[Thm.~1.1]{Fujino17}, i.e., that the pair be log canonical.  Since $(\mathcal X,\mathbf \Delta)$ is an snc pair, we have that   $(X,R+\Delta)$ is log canonical (\Cref{R:lc-pair}), and so we may indeed employ \cite[Thm.~7.3]{CKT21}  to conclude that  $K_{X/Z}+R+\Delta$ is pseudo-effective.
\end{proof}

We now review  some of the terminology from \cite{CPFol19} regarding ramification and branch divisors for morphisms in the situation of \eqref{E:CKT-cm}.  We have an exact sequence of sheaves 
\begin{equation}\label{E:TX/Z}
0\to {\mathcal T}_{\mathcal X/Z}\to {\mathcal T}_{\mathcal X}\to p^*{\mathcal T}_Z \to \mathcal R\to 0
\end{equation}
where $\mathcal T_{\mathcal X/Z}$ was defined in \Cref{D:AlgFol} to be the kernel of the morphism $\mathcal T_{\mathcal X}\to p^*\mathcal T_{\mathcal Z}$, and we are defining $\mathcal R$ to be the co-kernel.  By assumption, $\mathcal R$ is torsion and is  the quotient of the  torsion-free sheaf (vector bundle) $p^*\mathcal T_Z$, and so the determinant $\det \mathcal R$ is a line bundle admitting a nonzero section with zero locus contained in the support of $\mathcal R$ (e.g., \cite[Lem.~1.1]{CMZpositivity}). We define 
\begin{equation}\label{E:D:R(p)}
\mathbf R(p)
\end{equation}
 to be the zero divisor of this section, and call it  the ramification divisor of $p$ (see also  \cite[Def.~2.16 and Lem.~2.31]{CKT21} and note that in \cite[Def.~3.1]{CPFol19}, they denote $\mathbf R(p)$ by $D(p)$ if $p$ is not assumed to be equidimensional).

Taking determinants in \eqref{E:TX/Z}, we have
\begin{equation}\label{E:c1TX/Z}
c_1 ({\mathcal T}_{\mathcal X/Z})= -K_{\mathcal X}+p^*K_Z+\mathbf R(p)= -K_{\mathcal X/Z}+\mathbf R(p).
\end{equation}
We define 
\begin{equation}\label{E:D:DeltaHor}
\mathbf \Delta^{hor}\subseteq \mathbf \Delta
\end{equation}
 to be the union of the components of $\mathbf \Delta$ that map on to $Z$.

To state the next result, we recall the definition of an essentially equidimensional map: 
\begin{dfn}[Essentially equidimensional]\label{D:EssEqDim}
 For  $p:\mathcal{X}\dashrightarrow Z$  a rational map between a smooth separated integral DM stack $\mathcal X$ of finite type over $\mathbb C$ and a normal variety $Z$, we say
that $p$ is \emph{essentially equidimensional} if there exists an open substack $\mathcal{U}\subseteq \mathcal{X}$ with complement of codimension at least $2$ such that $p|_{\mathcal{U}}$ is an equidimensional morphism.
\end{dfn}

\begin{rem}\label{R:TXZR(p)DeltaHor}
 Note that for an equidimensional rational map  $p:\mathcal X\dashrightarrow Z$, by restricting to the substack $\mathcal U$ in the definition, and then extending over codimension-$2$ loci, we can define $\mathcal T_{\mathcal X/Z}$, $\mathbf R(p)$, and $\mathbf \Delta^{hor}$, using the definitions above for morphisms. 
\end{rem} 

We use \Cref{T:CKT21-7.3} to give the following generalization of a special case of \cite[Thm.~7.1]{CKT21} (see also \cite{CP15}):

\begin{teo}

\label{T:CKT21-7.1}
Let $\mathcal X$ be a smooth proper integral DM stack of finite type over $\mathbb C$ with projective coarse moduli space,
and let $p:\mathcal X\dashrightarrow Z$ be an essentially equidimensional 
rational map with connected fibers (on the locus of definition) to a normal projective variety $Z$. 
If $K_{\mathcal X/Z}+ \mathbf \Delta^{hor}$ 
 is pseudo-effective, then so is $K_{\mathcal X/Z} + \mathbf \Delta^{hor}
-\mathbf R(p)$.

\end{teo}

\begin{proof}
The first step is to construct the following diagram according to the procedure given in \cite[Pf.~Thm.~7.1, Step 1]{CKT21}:
\begin{equation}\label{E:CKT21-7.1diag}
\xymatrix@C=6em{
\widetilde {\mathcal X}  \ar[rr]^a_{\text{strong log res. of fiber product}} \ar[d]_{\tilde p}& & \overline{\mathcal{X}} \ar[r]^{b} \ar[d]_{\overline{p}} &\mathcal X  \ar@{.>}[d]
^p \\
\widetilde Z \ar[r]^{\sigma}_{\text{log res.}}& \widehat Z \ar[r]^\beta_{\stackrel{\text{strngly adpt. cov.}}{ (Z,\operatorname{OrbiBranch}(p))}} & Z \ar@{=}[r] &  Z
}
\end{equation}

To briefly review the construction, the orbifold branch divisor $\operatorname{OrbiBranch}(p)$ on $Z$ is defined in \cite[Def.~2.16]{CKT21}, and the
morphism $\beta:\widehat Z\to Z$ is a strongly adapted cover (\cite[Def.~2.37]{CKT21}) associated to the pair $(Z,\operatorname{OrbiBranch}(p))$, as constructed in \cite[Prop.~2.38]{CKT21}. The morphism $b:\overline{\mathcal{X}}\to \mathcal{X}$ is defined to be a strong schematic resolution of the indeterminacy locus of $p$ (see \Cref{P:KT23-4.4}).  
The morphism $\sigma:\widetilde Z\to \widehat Z$ is a  log resolution of the pair $(\widehat Z, \beta^*\operatorname{OrbiBranch}(p))$.  We then consider a strong log resolution of the fibered product $\widetilde Z\times _Z\mathcal X$ (i.e., it is an isomorphism over the smooth locus) and we obtain a generically finite
proper morphism $a:\widetilde {\mathcal X}\to \overline{\mathcal X}$, where $\widetilde {\mathcal X}$ is a smooth proper integral DM stack of finite type over $\mathbb C$ with projective coarse moduli space $\widetilde X$, 
and the union of the strict transform $\widetilde {\mathbf \Delta}$ of $\mathbf \Delta$, with the exceptional divisors of $a$, forms an snc divisor on $\widetilde {\mathcal X}$ 
(e.g.,  \cite[Thm.~1.3]{CMZslope_stability}).

Note that the fibers of $\tilde p:\widetilde {\mathcal X}\to \widetilde Z$ are connected.  Indeed, one can check this on the associated morphism $\widetilde X\to \widetilde Z$ on coarse moduli spaces.  One then has a morphism of normal projective varieties, with connected generic fiber.  Taking the Stein Factorization, which one can assume factors through a normal variety, and then using Zariski's main theorem, it follows that $\widetilde X\to \widetilde Z$ has connected fibers.  

An \'etale local computation carried out in \cite[Pf.~Thm.~7.1, Step 2--4]{CKT21}, which therefore carries over to the situation of DM stacks, then shows  (\cite[Consequence 7.8--9]{CKT21}) that there is an snc divisor $\widetilde {\mathbf \Delta}^{h}$ on $\widetilde {\mathcal X}$ such that 
\begin{equation}\label{E:CKT2178}
K_{\widetilde {\mathcal X}/\widetilde Z} + \widetilde {\mathbf \Delta}^{h}+\mathcal{E}''\sim_{\mathbb Q} (b\circ a)^*(K_{\mathcal X/Z}+\mathbf \Delta^{hor}-\mathbf R(p)),
\end{equation}
where $\mathcal{E}''$ is a $(b\circ a)$-exceptional divisor, and, moreover,  if $\widetilde {\mathcal F}$ is a general fiber of $\tilde p$, then the restriction $(K_{\widetilde {\mathcal X}/\widetilde Z} + \widetilde {\mathbf \Delta}^{h})|_{\widetilde {\mathcal F}}$ is pseudo-effective.  
By virtue of  \Cref{T:CKT21-7.3}, it follows  that $K_{\widetilde {\mathcal X}/\widetilde Z} + \widetilde {\mathbf \Delta}^{h}$ is pseudo-effective.  As $\mathcal{E}''$ is a $(b\circ a)$-exceptional divisor, it then follows from \eqref{E:CKT2178} and \cite[Cor.~3.19]{CMZslope_stability} that 
$K_{\mathcal X/Z}+\mathbf \Delta^{hor}-\mathbf R(p)$ is pseudo-effective.
\end{proof}

\section{Adapting Schnell's proof of the theorem of  Campana--P\u aun}\label{S:SchnellProof}

The goal of this section is to generalize \cite[Thm.~7.6]{CPFol19} to the setting of DM stacks (\Cref{T:SonCP-T1}); we follow the proof given in \cite[Thm.~1]{Schnell17Epi}.

\begin{teo}
\label{T:SonCP-T1}
Let $\mathcal X$ be a smooth proper integral  DM stack of finite type over $\mathbb C$ with projective coarse moduli space. Let $\mathbf \Delta\subseteq \mathcal X$ be a reduced divisor with at worst normal crossing singularities.
If some positive tensor power of $\Omega_{\mathcal X}^1(\log \mathbf \Delta)$ contains a subsheaf with big determinant, then $K_{\mathcal X}+\mathbf \Delta$ is big.
\end{teo}

\begin{rem}
Note that \cite{CPFol19} prove \Cref{T:SonCP-T1} under the additional assumption that $\mathcal X=[V/G]$ is the quotient of a smooth projective variety $V$ by a finite group $G$.  In fact, if $\mathcal X=[V/G]$, the result follows immediately from the case of smooth projective varieties by pull-back to $V$ via the finite \'etale morphism $V\to [V/G]$.
\end{rem}

The rest of this section is devoted to proving \Cref{T:SonCP-T1}.

\subsection{First reductions}\label{S:SonCPS2}
Here we follow the reduction in \cite[\S 2]{Schnell17Epi}.
 By assumption in  \Cref{T:SonCP-T1}, some positive tensor power $\Omega_{\mathcal X}^1(\log \mathbf \Delta)^{\otimes N}$ contains a subsheaf $\mathcal H$ with big determinant.  
 Let $\mathcal H^{\operatorname{sat}}$ be  the saturation of $\mathcal H$ in $\Omega_{\mathcal X}^1(\log \mathbf \Delta)^{\otimes N}$.  The claim is that $\mathcal H^{\operatorname{sat}}$ also has big determinant.  Indeed, we have a short exact sequence
 $$
 0\to \mathcal H\to \mathcal H^{\operatorname{sat}}\to \mathcal H^{\operatorname{sat}}/\mathcal H\to 0.
 $$
As the torsion sheaf on the right is by definition the quotient of a torsion-free sheaf, we may take determinants, and we have that 
$$
\det(\mathcal H^{\operatorname{sat}}) = \det (\mathcal H)\otimes \det ( \mathcal H^{\operatorname{sat}}/\mathcal H)
$$
is a big line bundle tensored by an effective line bundle  
\cite[Lem.~1.1]{CMZslope_stability}; 
therefore  we have that $\det(\mathcal H^{\operatorname{sat}})$ is big 
(\cite[Lem.~2.6]{CMZpositivity}). 

Thus, without loss of generality, we may assume that $\mathcal H$ is saturated in $\Omega_{\mathcal X}^1(\log \mathbf \Delta)^{\otimes N}$.
Let $\mathcal Q$ denote the quotient sheaf, which is now torsion-free by assumption, and consider the
resulting short exact sequence 
\begin{equation}\label{E:Rdx-ses-L}
0\to \mathcal H\to \Omega^1_{\mathcal X}(\log \mathbf \Delta)^{\otimes N} \to \mathcal Q\to 0.
\end{equation}

Since $K_{\mathcal X}+\mathbf \Delta$ 
represents the first Chern class of $\Omega^1_{\mathcal X}(\log \mathbf \Delta)$, we obtain 
$$
N\cdot (\dim \mathcal X)^{N-1} \cdot (K_{\mathcal X}+\mathbf \Delta)=c_1(\mathcal H)+c_1(\mathcal Q)
$$
in $\operatorname{CH}^1(\mathcal X)_{\mathbb Q}$.
Note that in our discussion here, for a sheaf $\mathcal F$ on $\mathcal X$ that is the quotient of a torsion-free sheaf, we will, for brevity, write $c_1(\mathcal F)$ when we mean the class of $c_1(\mathcal F)\cap [\mathcal X]$ in  $\operatorname{CH}^1(\mathcal X)_{\mathbb Q}$.

By assumption, the class $c_1(\mathcal H)$ is big; \Cref{T:SonCP-T1} will therefore be proved if we manage to show that the class
$c_1(\mathcal Q)$ is pseudo-effective. In fact, we are going to prove the following more general result, which is a generalization of \cite[Thm.~7.6, Thm.~1.2]{CPFol19} to the case of DM stacks.

\begin{teo}

\label{T:SonCP-T4}
 Let $\mathcal X$ be a smooth proper integral DM stack over $\mathbb C$ with projective coarse moduli space,  and let $ \mathbf \Delta \subseteq \mathcal X$ be a reduced divisor with at worst
normal crossing singularities. Suppose that some positive tensor power of  $\Omega^1_{\mathcal X}(\log \mathbf \Delta)$  contains a subsheaf
with big determinant. Then the first Chern class of every torsion-free quotient sheaf of every (positive) tensor power of $\Omega^1_{\mathcal X}(\log \mathbf \Delta)$ is  pseudo-eﬀective.
\end{teo}

\begin{rem}
Note that \cite{CPFol19} prove \Cref{T:SonCP-T4} under the additional assumption that $\mathcal X=[V/G]$; i.e., that $\mathcal X$ is the quotient of a smooth projective variety $V$ by a finite group $G$.  In fact, if $\mathcal X=[V/G]$, the result follows immediately from the case of smooth projective varieties by pull-back to $V$ via the finite \'etale morphism $V\to [V/G]$.
\end{rem}

\begin{rem}
Let $\mathcal Q$ denote a quotient of a  tensor power of  $\Omega^1_{\mathcal X}(\log \mathbf \Delta)$  in  \Cref{T:SonCP-T4}.
The reason for the assumption in \Cref{T:SonCP-T4} that  $\mathcal Q$ be torsion-free is simply because we do not know if the torsion subsheaf $\mathcal Q_{\operatorname{tor}}\subseteq \mathcal Q$    is the quotient of a torsion-free sheaf, and so we do not have a well-defined determinant for $\mathcal Q_{\operatorname{tor}}$.  Assuming that  $\mathcal Q_{\operatorname{tor}}$ is the quotient of a torsion-free sheaf (e.g., if the stabilizer at the generic point of $\mathcal X$ is trivial \cite[Thm.~1.2]{totaro_2004}), then it follows immediately from   \Cref{T:SonCP-T4} applied to $\mathcal Q/\mathcal Q_{\operatorname{tor}}$, and basic properties of determinants,  that $c_1(\mathcal Q)=c_1(\mathcal Q_{\operatorname{tor}})+ c_1(\mathcal Q/\mathcal Q_{\operatorname{tor}})$ is pseudo-effective.  
\end{rem}

\subsection{A few observations for the proof}

We start with  the following observation, which generalizes \cite[Lem.~3]{Schnell17Epi}:

\begin{lem}

\label{L:SonCP-L3}

If for some positive integer $N$ the tensor product $\Omega^1_{\mathcal X}(\log \mathbf \Delta)^{\otimes N}$ contains a subsheaf 
 of generic rank $r\ge 1$ and with big determinant, then
 $\Omega^1_{\mathcal X}(\log \mathbf \Delta)^{\otimes Nr}$ 
contains a big line bundle.
\end{lem}

\begin{proof}
The proof is identical to \cite[Lem.~3]{Schnell17Epi}; we include the proof here for completeness.  Let 
$\mathcal H\subseteq \Omega^1_{\mathcal X}(\log \mathbf \Delta)^{\otimes N}$ 
 be a subsheaf of generic rank $r \ge 1$, with the property that $\det \mathcal H$ is
big. As we saw in \S \ref{S:SonCPS2}, after replacing $\mathcal H$ by its saturation, whose determinant is  still big, we may assume that
the quotient sheaf
$
\mathcal Q:=\Omega^1_{\mathcal X}(\log \mathbf \Delta)^{\otimes N}/\mathcal H
$
is torsion-free, hence locally free outside a closed substack $\mathcal Z\subseteq \mathcal X$ of codimension $\ge 2$. On the open substack $\mathcal X-\mathcal Z$ of $\mathcal X$,
we have an inclusion of locally free sheaves
$$
\det \mathcal H \hookrightarrow \mathcal H^{\otimes r}\hookrightarrow \Omega^1_{\mathcal X}(\log \mathbf \Delta)^{\otimes Nr},
$$
which remains valid on $\mathcal X$ by Hartog's theorem.
\end{proof}

In other words, we are always free to assume in \Cref{T:SonCP-T1}  and \Cref{T:SonCP-T4} that the big subsheaf of the positive tensor power of $\Omega^1_{\mathcal X}(\log \mathbf \Delta)$ is a line bundle.   
We also make the following observation following \cite[Prop.~13]{Schnell17Epi}:

\begin{pro}

\label{P:SonCP-P13}
 Let $\mathcal X$ be a smooth proper DM stack over $\mathbb C$ with projective coarse moduli space,  and let $ \mathbf \Delta \subseteq \mathcal X$ be a reduced divisor with at worst
normal crossing singularities. Suppose that some positive tensor power of  $\Omega^1_{\mathcal X}(\log \mathbf \Delta)$  contains a subsheaf
with big determinant.  
If  \Cref{T:SonCP-T1} is true in dimension \emph{less than $\dim {\mathcal X}$}, and there is a surjective morphism $p:\mathcal X\to Z$ to a smooth projective variety $Z$ of dimension at least $1$, 
 then 
 $$
 K_{{\mathcal X}/Z}+\mathbf \Delta^{hor}
 $$
 is pseudo-effective.
\end{pro}

\begin{proof} The proof is by induction on the dimension of $ {\mathcal X}$, with the case that $ {\mathcal X}$ is of dimension $0$  being obvious.
Let $ {\mathcal X_z}$ be a general fiber of the morphism $ {p}:  {\mathcal X} \to Z$; since $\dim Z \ge 1$, 
we have $\dim  {\mathcal X_z} < 
\dim  {\mathcal X}$. Denote by $\mathbf \Delta_{ {\mathcal X_z}}$ the restriction of $\mathbf \Delta$; since $ {\mathcal X_z}$ is a general fiber, $\mathbf \Delta_{ {\mathcal X_z}}$ is still a normal crossing
divisor. Clearly
$$
(K_{ {\mathcal X}/Z}+\mathbf \Delta^{hor})|_{ {\mathcal X_z}}=K_{ {\mathcal X_z}}+\mathbf \Delta_{ {\mathcal X_z}}, 
$$
and according to \Cref{T:CKT21-7.3}, the pseudo-effectivity of $K_{ {\mathcal X}/Z}+\Delta^{hor}$ will follow if we manage to show that $K_{ {\mathcal X_z}}+\mathbf \Delta_{ {\mathcal X_z}}$ is pseudo-effective. 

By hypothesis (and \Cref{L:SonCP-L3}), there is a nonzero morphism
$$
\mathcal L\to \Omega^1_{ {\mathcal X}}(\log \mathbf \Delta)^{\otimes k}
$$
from a big line bundle $\mathcal L$ to some tensor power of $\Omega^1_{ {\mathcal X}}(\log \mathbf \Delta)$. Since $ {\mathcal X_z}$ is a general fiber of the morphism $ {p}:  {\mathcal X} \to Z$,
we can restrict this morphism to $ {\mathcal X_z}$  to obtain a nonzero morphism
$$
\mathcal L_{ {\mathcal X_z}}\to \left(\Omega^1_{ {\mathcal X}}(\log \mathbf \Delta)|_{ {\mathcal X_z}}\right)^{\otimes k}.
$$
Here $\mathcal L_{ {\mathcal X_z}}$ denotes the restriction of $\mathcal L$ to the fiber; since $\mathcal L$ is big and $ {\mathcal X_z}$ is general, $\mathcal L_{ {\mathcal X_z}}$ is also big.

The inclusion of $ {\mathcal X_z}$ into $ {\mathcal X}$ gives rise to a short exact sequence
$$
0\to \mathcal N^\vee_{ {\mathcal X_z}/ {\mathcal X}} \to \Omega^1_{ {\mathcal X}}(\log \mathbf \Delta)|_{ {\mathcal X_z}} \to \Omega^1_{ {\mathcal X_z}}(\log \mathbf \Delta_{ {\mathcal X_z}}) \to 0
$$
which induces a filtration on the $k$-th tensor power of the locally free sheaf in the middle. Since the
co-normal bundle $\mathcal N^\vee_{ {\mathcal X_z}/ {\mathcal X}}$  is trivial of rank $\dim Z$, we find, by looking at the subquotients of this filtration,
that there is a nonzero morphism
\begin{equation}\label{E:SonCP-EP13}
\mathcal L_{ {\mathcal X_z}} \to  \Omega^1_{ {\mathcal X_z}}(\log \mathbf \Delta_{ {\mathcal X_z}})^{\otimes j}
\end{equation}
for some $0\le j\le k$.  Because $\mathcal L_{ {\mathcal X_z}} $ is big, we actually have $1\le j\le k$.  
Since we are assuming by induction that  \Cref{T:SonCP-T1} is true for the pair $( {\mathcal X_z},\mathbf \Delta_{ {\mathcal X_z}})$, the class $K_{ {\mathcal X_z}}+\mathbf \Delta_{ {\mathcal X_z}}$ is big on $ {\mathcal X_z}$, and hence pseudo-effective.  Appealing to \Cref{T:CKT21-7.3}, we deduce that the class $K_{ {\mathcal X}/Z}+\Delta^{hor}$ is pseudo-effective on $ {\mathcal X}$.
\end{proof}

The rest of this section is devoted to proving \Cref{T:SonCP-T4}.
In fact, following \cite{Schnell17Epi}, we will prove \Cref{T:SonCP-T4} and \Cref{T:SonCP-T1} simultaneously via induction on $\dim \mathcal X$.

\subsection{Set-up for proof by induction and contradiction}\label{S:SContra}

We now begin to  prove  \Cref{T:SonCP-T1} and \Cref{T:SonCP-T4}, following \cite{Schnell17Epi}.
The proof is by induction on $\dim \mathcal X$;  \Cref{T:SonCP-T1} and \Cref{T:SonCP-T4} are obvious if $\dim \mathcal X=0$.  So assume we have proven  \Cref{T:SonCP-T1} and \Cref{T:SonCP-T4} up to dimension $\dim \mathcal X-1$.  To prove both theorems for $\mathcal X$, it suffices to prove \Cref{T:SonCP-T4} for $\mathcal X$.  We proceed   to prove \Cref{T:SonCP-T4} by 
contradiction:
\begin{quote}
\emph{Suppose, for the sake of contradiction, that, for some
integer $N\ge 1$, and for some non-zero torsion-free quotient sheaf $\mathcal Q$ of 
$\Omega^1_{\mathcal X}(\log \mathbf \Delta)^{\otimes N}$, 
the class $c_1(\mathcal Q)$ were not pseudo-effective.}
\end{quote}

\subsection{Slopes and foliations}
By  the definition of 
 a movable class $\alpha\in \operatorname{N}_1(\mathcal X)_{\mathbb R}$ 
 (see \cite[Def.~2.1]{CMZslope_stability}),  assuming the existence of $\mathcal Q$ as above, i.e., $\mathcal Q$ is not pseudo-effective, there would be a movable class $\alpha$ 
  such that 
 $c_1(\mathcal Q)\cdot \alpha<0$.  
We therefore have that the $\alpha$-slope, $\mu_\alpha(\mathcal Q)$,  satisfies
$$
\mu_\alpha(\mathcal Q)=\frac{c_1(\mathcal Q)\cdot \alpha}{\operatorname{rk}\mathcal Q}<0,
$$
and so $\mathcal Q$ is a torsion-free quotient sheaf of $\Omega^1_{\mathcal X}(\log \mathbf \Delta)^{\otimes N}$ with negative $\alpha$-slope. 
 The dual sheaf $\mathcal Q^\vee$ is
therefore a 
subsheaf of ${\mathcal T}_{\mathcal X}(-\log \mathbf \Delta)^{\otimes N}$ with positive $\alpha$-slope, where we define ${\mathcal T}_{\mathcal X}(-\log \mathbf \Delta):=(\Omega^1_{\mathcal X}(\log \mathbf \Delta))^\vee$.  

The fact that  ${\mathcal T}_{\mathcal X}(-\log \mathbf \Delta)^{\otimes N}$ has a subsheaf of positive $\alpha$-slope implies by 
\cite[Cor.~3.26]{CMZslope_stability}
that ${\mathcal T}_{\mathcal X}(-\log \mathbf \Delta)$ has a subsheaf of positive $\alpha$-slope.  Define  $\mathcal F_{\mathbf \Delta}\subseteq {\mathcal T}_{\mathcal X}(-\log \mathbf \Delta)$ to be the maximal $\alpha$-destabilizing subsheaf 
\cite[Prop.~3.7]{CMZslope_stability}.
The following generalizes \cite[Lem.~6]{Schnell17Epi}:

\begin{lem}

\label{L:SonCP-L6}
The sheaf $ \mathcal F_{\mathbf \Delta}$ is a saturated  $\alpha$-semistable subsheaf of ${\mathcal T}_{\mathcal X}(-\log \mathbf \Delta)$, of positive $\alpha$-slope. Every non-zero 
subsheaf of ${\mathcal T}_{\mathcal X}(-\log \mathbf \Delta)/\mathcal F_{\mathbf \Delta}$  has $\alpha$-slope less than $\mu_\alpha(\mathcal F_{\mathbf \Delta})$. 
\end{lem}

\begin{proof}
This is clear from the definitions 
(see \cite[Prop.~3.7 and Lem.~3.1]{CMZslope_stability}).
\end{proof}

Recall that we have an inclusion ${\mathcal T}_{\mathcal X}(-\log \mathbf \Delta)\subseteq \mathcal T_{\mathcal X}$ obtained by dualizing the natural inclusion $\Omega^1_{\mathcal X}\hookrightarrow \Omega^1_{\mathcal X}(\log \mathsf \Delta)$. We define another coherent subsheaf $\mathcal F\subseteq \mathcal T_{\mathcal X}$
as the saturation of $\mathcal F_{\mathbf \Delta}$ in $\mathcal T_{\mathcal X}$; then $\mathcal T_{\mathcal X}/\mathcal F$ is torsion-free, and
\begin{equation}\label{E:SonCP-eq3.2}
\mathcal F\cap \mathcal T_{\mathcal X}(-\log \mathbf \Delta)=\mathcal F_{\mathbf \Delta}.
\end{equation}
We will see in a moment that $\mathcal F$ is actually a (possibly singular) foliation on $\mathcal X$. Recall that, in
general, a foliation on a smooth DM stack  is a saturated subsheaf $\mathcal F \subseteq \mathcal T_{\mathcal X}$ that is closed
under the Lie bracket on $\mathcal T_{\mathcal X}$. 

To show that $\mathcal F$ is a foliation, consider that from the Lie bracket, one constructs an $\mathcal O_{\mathcal X}$-linear mapping
$$
N:\mathcal F\hat\otimes \mathcal F\to \mathcal T_{\mathcal X}/\mathcal F
$$
called the O'Neil tensor of $\mathcal F$; evidently, $\mathcal F$ is a foliation if and only if its O'Neil tensor vanishes. The following generalizes \cite[Lem.~7]{Schnell17Epi}: 

\begin{lem}

\label{L:SonCP-L7}
The O'Neil tensor $N:\mathcal F\hat \otimes \mathcal F\to \mathcal F_{\mathcal X}/\mathcal F$ 
vanishes, and $\mathcal F$ is therefore a foliation on $\mathcal X$.
\end{lem}

\begin{proof}
The proof is identical to that of \cite[Lem.~7]{Schnell17Epi}; we include the details here for completeness.
The Lie bracket of two sections of $\mathcal T_{\mathcal X}(-\log \mathbf \Delta)$ is a section of $\mathcal T_{\mathcal X}(-\log \mathbf \Delta)$ (this can be checked locally, and so follows from the case of varieties), and so we get a
log O'Neil tensor
$$
N_{\mathbf\Delta}:\mathcal F_{\mathbf \Delta}\hat\otimes \mathcal F_{\mathbf \Delta}\longrightarrow \mathcal T_{\mathcal X}(-\log \mathbf \Delta)/\mathcal F_{\mathbf \Delta}.
$$
The key point is that $N_{\mathbf \Delta}  = 0$. Indeed, by 
\cite[Cor.~3.25]{CMZslope_stability}, 
the tensor product modulo torsion, $\mathcal F_{\mathbf \Delta}\hat \otimes \mathcal F_{\mathbf \Delta}$,  
is again $\alpha$-semistable of slope
$$
\mu_\alpha(\mathcal F_{\mathbf \Delta}\hat \otimes \mathcal F_{\mathbf \Delta})=2\cdot \mu_\alpha(\mathcal F_{\mathbf \Delta})>\mu_\alpha(\mathcal F_{\mathbf \Delta}),
$$
which is strictly greater than the slope of the image of $N_{\mathbf \Delta}$, if  the image  is nonzero  
(we have  by \Cref{L:SonCP-L6} that $\mu_\alpha(\mathcal F_{\mathbf \Delta})$ is greater than the slope of any nonzero subsheaf of $\mathcal T_{\mathcal X}(-\log \mathbf \Delta)/\mathcal F_{\mathbf \Delta}$).
This inequality among slopes implies that $N_{\mathbf \Delta} = 0$ by  
\cite[Prop.~3.2]{CMZslope_stability}.

The O'Neil tensor $N$ and the log O'Neil tensor $N_{\mathbf \Delta}$ are both induced by the Lie bracket
on $\mathcal T_{\mathcal X}$, and so we have the following commutative diagram:
$$
\xymatrix@C=3em{
\mathcal F_{\mathbf \Delta}\hat \otimes \mathcal F_{\mathbf \Delta} \ar[r]^<>(0.5){N_{\mathbf \Delta}} \ar[d]& \mathcal T_{\mathcal X}(-\log \mathbf \Delta)/\mathcal F_{\mathbf \Delta} \ar@{^(->}[d]\\
\mathcal F\hat \otimes \mathcal F \ar[r]^N& \mathcal T_{\mathcal X}/\mathcal F\\
}
$$
The vertical arrow on the right is injective by \eqref{E:SonCP-eq3.2}. Now $N_{\mathbf \Delta}=0$ implies that $N$ factors through
the cokernel of the vertical arrow on the left; but the cokernel is a torsion sheaf, whereas $\mathcal T_{\mathcal X}/\mathcal F$ is
torsion-free. The conclusion is that $N = 0$. 
\end{proof}

We now take a moment to investigate the quotient $\mathcal F/\mathcal F_{\mathbf \Delta}$, which we will use in several computations that follow.  
It is easy to see from \eqref{E:SonCP-eq3.2} that we have an inclusion of sheaves
\begin{equation}\label{E:F/FDeltaInc}
\mathcal F/\mathcal F_{\mathbf \Delta}\hookrightarrow \mathcal T_{\mathcal X}/\mathcal T_{\mathcal X}(-\log \mathbf \Delta).
\end{equation}
The sheaf on the right-hand side is a torsion sheaf supported on the divisor $\mathbf \Delta$, and a brief computation (which can be reduced to the case of smooth varieties, and is established in the proof of \cite[Lem.~12]{Schnell17Epi}) shows that
\begin{equation}\label{E:SonCP-N1}
\mathcal T_{\mathcal X}/\mathcal T_{\mathcal X}(-\log \mathbf \Delta) \cong \bigoplus _{\mathbf D\subseteq \mathbf \Delta} \iota_{\mathbf D,*}\mathcal N_{\mathbf D/\mathcal X}
\end{equation}
is isomorphic to the direct sum of the normal bundles of the irreducible components $\mathbf D$ of $\mathbf \Delta$; we denote by $\iota_{\mathbf D}:\mathbf D\hookrightarrow \mathcal X$ the inclusion. 
We have the usual identification $\mathcal N_{\mathbf D/\mathcal X}\cong \mathcal O_{\mathbf D}(\mathbf D)$, 
and, from the short exact sequence 
$0\to \mathcal O_{\mathcal X}\to \mathcal O_{\mathcal X}(\mathbf D)\to \iota_{\mathbf D,*}\mathcal O_{\mathbf D}(\mathbf D)\to 0$, 
we have that  $\iota_{\mathbf D,*}\mathcal O_{\mathbf D}(\mathbf D)$ is the quotient of a torsion-free coherent sheaf on $\mathcal X$, and $c_1(\iota_{\mathbf D,*}\mathcal O_{\mathbf D}(\mathbf D))\cong \mathcal O_{\mathcal X}(\mathbf D)$. 

From the inclusion \eqref{E:F/FDeltaInc} and the isomorphism \eqref{E:SonCP-N1}, we can conclude that the rank
of $\mathcal  F /\mathcal F_{\mathbf \Delta}$  at the generic point of $\mathbf D$ is either $0$ or $1$, and so 
\begin{equation}\label{E:SonCP-c1-1}
c_1(\mathcal F/\mathcal F_{\mathbf \Delta}) =c_1 \left(\mathcal O_{\mathcal X}(\sum_{\mathbf D\subseteq \mathbf \Delta}a_{\mathbf D}\mathbf D)\right)
\end{equation}
where $a_{\mathbf D} = 0$ if $\mathcal F= \mathcal F_{\mathbf \Delta}$ at the generic point of $\mathbf D$, and $a_{\mathbf D} = 1$, otherwise.  Note that we have a well-defined first Chern class of the torsion sheaf  $\mathcal F/\mathcal F_\Delta$, since by construction it is a quotient of the torsion-free sheaf  $\mathcal F$. 

From this we can show the following:
\begin{lem}\label{L:Fpos_mu}
Every quotient sheaf of $\mathcal F$ has positive $\alpha$-slope. 
\end{lem}

\begin{proof}
Let $\mathcal G$ be a quotient sheaf of $\mathcal F$; i.e., we have a surjection $\mathcal F\twoheadrightarrow \mathcal G$.  Consider the following commutative diagram of short exact sequences
$$
\xymatrix{
0  \ar[r]& \mathcal F_{\mathbf \Delta} \ar[r]\ar@{->>}[d]& \mathcal F \ar[r]\ar@{->>}[d]& \mathcal F/\mathcal F_{\mathbf \Delta}\ar[r]\ar@{->>}[d]& 0\\
0 \ar[r]&\mathcal G'\ar[r]& \mathcal G \ar[r]& \mathcal G/\mathcal G'\ar[r]& 0
}
$$
where $\mathcal G'$ is defined to be the image of $\mathcal F_\Delta$ in $\mathcal G$. As $\mathcal F_{\mathbf \Delta}$ 
is $\alpha$-semistable with $\mu_\alpha (\mathcal F ) >0$, it follows that every quotient sheaf of $\mathcal F_{\mathbf \Delta}$ has positive $\alpha$-slope.  In particular, $\mu_\alpha (\mathcal G')>0$.  As $\mathcal F/\mathcal F_{\mathbf \Delta}$ is torsion, it follows that $\mathcal G/\mathcal G'$ is, as well; note that both are quotients of the torsion-free sheaf $\mathcal F$.  Since, for torsion sheaves that are quotients of torsion-free sheaves,  the first Chern class is effective, we have that $\mu_\alpha (\mathcal G)\ge \mu_{\alpha}(\mathcal G')>0$, completing the proof. 
\end{proof}

It follows from \Cref{L:Fpos_mu} and  \Cref{T:SonCP-T8} that
the foliation $\mathcal F$ is algebraic (\Cref{D:AlgFol}). In other words, there exists a dominant rational mapping
$$
p:\mathcal X\dashrightarrow Z
$$
to a smooth projective variety $Z$, and using the abuse of notation introduced in  \Cref{R:AlgFol}, we say 
$$
\mathcal F = \ker \left(Tp : {\mathcal T}_{\mathcal X} \to p^*{\mathcal T}_Z\right)
$$
generically, and we write $\mathcal F={\mathcal T}_{\mathcal X/Z}$.
More precisely, 
  the rational map $p$ can be taken to be a morphism on an open substack $\mathcal U\subseteq \mathcal X$ with complement of codimension $\ge 2$, and we have that $\mathcal F|_{\mathcal U}= \ker \left(Tp|_{\mathcal U} : {\mathcal T}_{\mathcal X}|_{\mathcal U} \to p|_{\mathcal U}^*{\mathcal T}_Z\right)$.
Moreover, we recall that $\mathcal F$, being a foliation, is reflexive, as 
a saturated subsheaf of a locally free sheaf is reflexive  (e.g., \cite[Prop.~1.1]{Hart80}).

In order to make our inductive argument, we will want to know that 
 $\dim Z\ge 1$; the following generalizes \cite[Lem.~10]{Schnell17Epi}: 
 
 \begin{lem}
 
 \label{L:SonCP-L10}
In the notation above, we have $\dim Z\ge 1$. 
\end{lem}

\begin{proof} The proof is essentially identical to \cite[Lem.~10]{Schnell17Epi}, but we include the details for convenience.
 If $\dim Z = 0$, then $\mathcal F= \mathcal T_{\mathcal X}$  and $\mathcal F_{\mathbf \Delta} = \mathcal T_{\mathcal X}(-\log \mathbf \Delta)$, so that,  consequently, the logarithmic tangent bundle $\mathcal T_{\mathcal X}(-\log \mathbf \Delta)$ is $\alpha$-semistable of positive slope (\Cref{L:SonCP-L6}). Since the dual 
 (\cite[Cor.~3.11]{CMZslope_stability}) 
 and tensor product 
  (\cite[Thm.~3.23]{CMZslope_stability}) 
 of $\alpha$-semistable sheaves remains $\alpha$-semistable, this means that any tensor power
of $\Omega^1_{\mathcal X}(\log \mathbf \Delta)$
 is $\alpha$-semistable of negative slope. But that contradicts the hypothesis of \Cref{T:SonCP-T4},
namely that some tensor power of  $\Omega^1_{\mathcal X}(\log \mathbf \Delta)$ contains a subsheaf with big determinant, because the
$\alpha$-slope of such a subsheaf is obviously non-negative.  
\end{proof}

We consider a  resolution of the rational map $p$ as in \Cref{P:KT23-4.4} and the proof of \Cref{T:CKT21-7.1} (see diagram \eqref{E:CKT21-7.1diag})
\begin{equation}\label{E:res-for-contra}
\xymatrix{
&\overline  {\mathcal X} \ar[rd]^{ \overline{p}} \ar[ld]_b&\\
\mathcal X\ar@{-->}[rr]^p&&Z.
}
\end{equation}
Note that  $\overline{\mathcal X}$ is a smooth proper integral DM stack over $\mathbb C$ with projective coarse moduli space, and $b:\overline {\mathcal X}\to \mathcal X$ is birational.
 After further blow-ups 
 (e.g.,  \cite[Thm.~1.3]{CMZslope_stability}),  
 we may assume also that both $K_{\overline {\mathcal X}/\mathcal X}$ and $b^*\mathbf \Delta$ are normal crossings divisors.  By the constructions in \Cref{P:KT23-4.4} and  the  proof of \Cref{T:CKT21-7.1}, we may assume $b$ is an isomorphism over the open subset where $p$ is already a morphism.

\subsection{Establishing the contradiction}\label{S:Contradiction}

Let
$\bar {\mathbf \Delta}$  be the reduced normal crossing divisor whose support is equal to the pre-image of $\mathbf \Delta$  in $\overline {\mathcal X}$.  
Recall  from \eqref{E:D:DeltaHor}  that the horizontal part $\bar{\mathbf \Delta}^{hor}\subseteq \bar{\mathbf \Delta}$ is the union of all irreducible components of $\bar{\mathbf \Delta}$  that map
onto $Z$; evidently, $\bar{\mathbf \Delta}^{hor}$  is again a reduced divisor on $\overline{\mathcal X}$ with at worst normal crossing singularities.
Let $\mathbf R(\overline{p})$ denote the ramification divisor of the morphism $\overline{p}: \overline{\mathcal X} \to Z$ (see \eqref{E:c1TX/Z}).
We also have $\mathbf \Delta^{hor}$ and $\mathbf R(p)$ on $\mathcal X$; see \Cref{R:TXZR(p)DeltaHor}. 
Note also that the numerical pull-back    $$\bar \alpha :=b^*\alpha,$$  of the movable class $\alpha$ on $\mathcal X$,  is a  movable class on $\overline {\mathcal X}$ 
 (\cite[Lem.~2.5(2)]{CMZslope_stability}).

\begin{lem}\label{L:Spock-not-Spock}
In the notation above:
\begin{enumerate}[label=(\alph*)]
\item \label{E:L:Spock}

If $K_{\overline{\mathcal X}/Z}+\bar{\mathbf \Delta}^{hor}$ is pseudo-effective, then $K_{\mathcal X/Z}+\mathbf \Delta^{hor}-\mathbf R(p)$ is pseudo-effective.

\item \label{E:L:not-Spock}
If  $(K_{\overline{\mathcal X}/Z}+\bar{\mathbf \Delta}^{hor}-\mathbf R(\overline{p}))\cdot \overline{\alpha}<0$, then $K_{\mathcal X/Z}+\mathbf \Delta^{hor}-\mathbf R(p)$ is not pseudo-effective. 
\end{enumerate}
\end{lem}

\begin{proof}
\ref{E:L:Spock} 
The key point is that $K_{\overline{\mathcal X}/Z}+\mathbf \Delta^{hor}$ and $K_{\mathcal X/Z}+\mathbf \Delta^{hor}$ are isomorphic over the complement of the locus contracted by  $b:\overline {\mathcal X}\to \mathcal X$.   It then follows from  \cite[Cor.~3.19]{CMZslope_stability} that if $K_{\overline{\mathcal X}/Z}+\mathbf \Delta^{hor}$ is pseudo-effective, then so is $K_{\mathcal X/Z}+\mathbf \Delta^{hor}$.  Now using that $p$ is essentially equidimensional (note that in general $\overline p$ will not be!), one can invoke \Cref{T:CKT21-7.1}, and conclude that $K_{\mathcal X/Z}+\mathbf \Delta^{hor}-\mathbf R(p)$ is pseudo-effective.

\ref{E:L:not-Spock} 
Again, the key point is that $K_{\overline{\mathcal X}/Z}+\mathbf \Delta^{hor}-\mathbf R(\overline{p})$ and $K_{\mathcal X/Z}+\mathbf \Delta^{hor}-\mathbf R(p)$ are isomorphic over the complement of the locus contracted by  $b:\overline {\mathcal X}\to \mathcal X$.  
Then it follows  from \cite[Lem.~3.18(b)]{CMZslope_stability}   that 
$0>(K_{\overline{\mathcal X}/Z}+\bar{\mathbf \Delta}^{hor}-\mathbf R(\overline{p}))\cdot \overline{\alpha}=(K_{\mathcal X/Z}+\mathbf \Delta^{hor}-\mathbf R(p))\cdot \alpha$.
\end{proof}

To obtain our contradiction, we start with the following lemma:

\begin{lem}\label{L:I-am-Spock}
In the notation above, $K_{\overline{\mathcal X}/Z}+\bar {\mathbf \Delta}^{hor}$ is pseudo-effective.
\end{lem}

\begin{proof}
We have 
$$
\Omega^1_{\overline {\mathcal X}}(\log \bar {\mathbf \Delta}) = b^*\Omega^1_{\mathcal X}(\log \mathbf \Delta),
$$
and, since the pullback of a big line bundle 
 by the morphism $b$ stays big 
  (\cite[Lem.~2.5]{CMZpositivity}), 
 and we know that $\Omega^1_{\mathcal X}(\log \mathbf \Delta)$ contains a big line bundle (\Cref{L:SonCP-L3}), it is still true that some tensor power of
$\Omega^1_{\overline {\mathcal X}}(\log \bar {\mathbf \Delta})$  contains a big line bundle (and therefore contains a subsheaf with big determinant).  
It then follows from \Cref{P:SonCP-P13}, that $K_{\overline{\mathcal X}/Z}+\bar {\mathbf \Delta}^{hor}$ is pseudo-effective, as we are assuming \Cref{T:SonCP-T1} is true in dimension less than $\dim {\mathcal X}$.
\end{proof}

To establish that we in fact have a contradiction, we show the following:

\begin{pro}\label{P:I-am-not-Spock}
In the notation above,  $(K_{\overline{\mathcal X}/Z}+\mathbf \Delta^{hor}-\mathbf R(\overline{p}))\cdot \overline{\alpha}<0$.
\end{pro}

\begin{proof}
To start, define
$$
\overline {\mathcal F}:=\mathcal T_{\overline {\mathcal X}/Z}=\ker \left( T\overline p : \mathcal T_{\overline {\mathcal X}} \to \overline p^*\mathcal T_Z\right)
$$
which is reflexive (use  \cite[Prop.~1.1]{Hart80} and the fact that $\mathcal T_{\overline {\mathcal X}}/\overline {\mathcal F}$, being a subsheaf of the locally free sheaf $\overline  p^*\mathcal T_Z$,  is torsion-free).
Then, with $\overline {\mathcal F}$ being a saturated subsheaf of $\mathcal T_{\overline {\mathcal X}}$, one can check that the intersection
$$
\overline {\mathcal F}_{\bar{\mathbf \Delta}}:=\overline {\mathcal F}\cap \mathcal T_{\overline {\mathcal X}}(-\log \bar {\mathbf \Delta})
$$
is a saturated subsheaf of $\mathcal T_{\overline {\mathcal X}}(-\log \bar {\mathbf \Delta})$ (this is local, and so reduces immediately to the case of varieties, which is established in the proof of \cite[Lem.~11]{Schnell17Epi}); therefore $\overline {\mathcal F}_{\mathbf \Delta}$ is a reflexive subsheaf of $ \mathcal T_{\overline {\mathcal X}}(-\log \bar {\mathbf \Delta})$.   One can also check that the pushforward of $\overline {\mathcal F}_{\mathbf \Delta}$ to $\mathcal X$ is $\mathcal F_{\mathbf \Delta}$, using \eqref{E:SonCP-eq3.2}, the fact that $\mathcal F_{\mathbf \Delta}$ is reflexive, and that the logarithmic tangent sheaf on $\overline{ \mathcal X}$ pushes forward to the logarithmic tangent sheaf on $\mathcal X$; see the proof of \cite[Lem.~11]{Schnell17Epi}.
Note that  using 
 \cite[Lem.~3.18(b)]{CMZslope_stability}, 
we have
$$
c_1(\overline {\mathcal F}_{{\mathbf \Delta}})\cdot \tilde\alpha =c_1(\overline {\mathcal F}_{ {\mathbf \Delta}})\cdot b^*\alpha  =c_1(\mathcal F_{\mathbf \Delta})\cdot  \alpha >0.
$$
 Recall from \eqref{E:c1TX/Z} the following formula for the first Chern
class of our foliation $\overline{\mathcal F} \subseteq \mathcal T_{\overline{\mathcal X}}$:
\begin{equation}\label{E:c1FF}
c_1(\overline{\mathcal F} ) = c_1(\mathcal T_{\overline{\mathcal X}/Z}) =-K_{\overline{\mathcal X}/Z} + \mathbf R(\overline{p}).
\end{equation}
Rather than computing the first Chern class of $\overline{\mathcal F}_{\bar{\mathbf \Delta}}$, we will 
use the fact that $\overline{\mathcal F}= \mathcal T_{\overline{\mathcal X}/Z}$ to estimate the difference
$$
c_1(\overline{\mathcal F})-c_1(\overline{\mathcal F}_{\bar{\mathbf \Delta}}) = c_1(\overline{\mathcal F}/\overline{\mathcal F}_{\bar{\mathbf \Delta}}).
$$
The following generalizes \cite[Lem.~12]{Schnell17Epi}:

\begin{lem}

\label{L:SonCP-L12}
 The class
$$
c_1(\overline{\mathcal F})-c_1(\overline{\mathcal F}_{\bar{\mathbf \Delta}})-\bar{\mathbf \Delta}^{hor}
$$
is effective.
\end{lem}

\begin{proof}
From equation \eqref{E:SonCP-c1-1}, 
it suffices to show  that $\overline{\mathcal F}\ne  \overline{\mathcal F}_{\bar{\mathbf \Delta}}$  exactly at the generic points of each of the irreducible
components of $\bar{\mathbf \Delta}^{hor}$. This is ultimately a consequence of the fact that $\overline{\mathcal F}= \mathcal T_{\overline{\mathcal X}/Z}$, and is a local computation that  reduces immediately to the case of varieties, which is done in the proof of \cite[Lem.~12]{Schnell17Epi}.  This completes the proof of \Cref{L:SonCP-L12}.
\end{proof}

One concludes that
\begin{equation}\label{E:SonCP-E4.5}
-\left(K_{\overline{\mathcal X}/Z}+\bar{\mathbf \Delta}^{hor}-\mathbf R(\overline{p})\right)\cdot \overline{\alpha} = \left(c_1(\overline{\mathcal F})-\bar{\mathbf \Delta}^{hor}\right)\cdot \overline{\alpha} \ge c_1(\overline{\mathcal F}_{\bar{\mathbf \Delta}})\cdot \overline{\alpha} >0,
\end{equation}
where the first equality is from \eqref{E:c1FF},  the first inequality is from  \Cref{L:SonCP-L12}, and the last inequality is \Cref{L:SonCP-L6}.  This completes the proof of \Cref{P:I-am-not-Spock}.
\end{proof}

\subsection{Proof of \Cref{T:SonCP-T1} and \Cref{T:SonCP-T4}}

\begin{proof}[Proof of \Cref{T:SonCP-T1} and \Cref{T:SonCP-T4}]
First, recall the general strategy of the proof from \S \ref{S:SContra}.  We are assuming we have proven \Cref{T:SonCP-T1} and \Cref{T:SonCP-T4} in dimension less than $\dim \mathcal X$, and, to complete the proof by induction, we are trying to prove that  \Cref{T:SonCP-T1} and \Cref{T:SonCP-T4} hold for $\mathcal X$.   We saw that it sufficed to prove 
\Cref{T:SonCP-T4} for $\mathcal X$.  We then proceeded to prove  
\Cref{T:SonCP-T4} for $\mathcal X$ by contradiction.  Under the assumption that \Cref{T:SonCP-T4} was false for $\mathcal X$ 
we showed that there exists a dominant rational map $p:\mathcal X\dashrightarrow Z$ to a smooth projective variety $Z$ of dimension at least $1$, as well as a resolution $\overline p:\overline{\mathcal X}\to Z$ of the rational map $p$, such that the following two things were were true.  
First, combining  \Cref{L:I-am-Spock} with \Cref{L:Spock-not-Spock}\ref{E:L:Spock}, we obtained that  the divisor  $K_{\mathcal X/Z}+\mathbf \Delta^{hor}-\mathbf R(p)$ was  pseudo-effective.
Second, combining  \Cref{P:I-am-not-Spock} with \Cref{L:Spock-not-Spock}\ref{E:L:not-Spock}, we obtained that   $K_{\mathcal X/Z}+\mathbf \Delta^{hor}-\mathbf R(p)$ was  not pseudo-effective.
This is a contradiction, and therefore \Cref{T:SonCP-T4} is true for $\mathcal X$, completing the proof.
\end{proof}

\begin{rem}
In light of the arguments in \S\ref{S:Contradiction}, one might try to prove \Cref{T:SonCP-T4} by reducing to the morphism $\overline  p:\overline {\mathcal X}\to Z$; i.e., replacing the rational map $p:\mathcal X\dashrightarrow Z$ with $\overline p:\overline {\mathcal X}\to Z$.  Indeed, we showed in \Cref{L:Spock-not-Spock}\ref{E:L:not-Spock} that $K_{\overline{\mathcal X}/Z}+\bar{\mathbf \Delta}^{hor}-\mathbf R(\overline{p})$ was not pseudo-effective, and so one might hope that one could obtain a contradiction by showing that $K_{\overline{\mathcal X}/Z}+\bar{\mathbf \Delta}^{hor}-\mathbf R(\overline{p})$ was pseudo-effective.  Certainly we showed in \Cref{L:Spock-not-Spock}\ref{E:L:Spock} that $K_{\overline{\mathcal X}/Z}+\bar{\mathbf \Delta}^{hor}$ is pseudo-effective.  Unfortunately, $\overline p$ may not be equidimensional, and so one cannot apply  \Cref{T:CKT21-7.1}  to conclude that  $K_{\overline{\mathcal X}/Z}+\bar{\mathbf \Delta}^{hor}-\mathbf R(\overline{p})$ is pseudo-effective.  
There are two approaches one might try to take working directly with $\overline p$: 
\begin{enumerate}
\item First, one might hope that we could use the fact that $K_{\mathcal X/Z}+\mathbf \Delta^{hor}-\mathbf R(p)$ was was pseudo-effective to show that $K_{\overline{\mathcal X}/Z}+\bar{\mathbf \Delta}^{hor}-\mathbf R(\overline{p})$ was pseudo-effective; unfortunately, the converse of   \cite[Cor.~3.19]{CMZslope_stability} may fail, which in this context means that it is hard to rule out there being a moving class on $\overline {\mathcal X}$ that is not the pull back of a moving class on $\mathcal X$, and which pairs negatively with $ K_{\overline{\mathcal X}/Z}+\bar{\mathbf \Delta}^{hor}-\mathbf R(\overline{p})$.

\item   
As a second approach, we believe that one could try to argue directly with $\overline p$ by extending various notions for smooth projective varieties to the case of stacks; e.g., make $\overline p$ into a  $\bar {\mathbf \Delta}$-neat fibration (see \cite[Def.~2.10]{CP15}) or a \emph{prepared} morphism (see \cite[1.1.3]{Campana2004Orbifolds})  and then follow the argument of \cite[Thm.~3.4]{CPFol19}, which utilizes the arguments of \cite[Thm.~2.11]{CP15} and   \cite[\S 4]{Campana2004Orbifolds}.  We felt the approach we took in the proof above is more direct. 
\end{enumerate}

\end{rem}


\ifArxiv

\subsection{Some related examples}
Here we recap and expand slightly on \cite[Rem.~14]{Schnell17Epi}.  
The point is that many  of the arguments for proving \Cref{T:SonCP-T1} and \Cref{T:SonCP-T4} (for example, the proof of \Cref{L:SonCP-L10}), go through when some
tensor power of $\Omega^1_{\mathcal X}(\log \mathbf \Delta)$ 
 contains a subsheaf with pseudo-effective determinant (i.e., rather than needing to make the stronger assumption that the subsheaf have big determinant).   Here we explain why one cannot weaken the hypotheses of  \Cref{T:SonCP-T1} and \Cref{T:SonCP-T4} in this way, even with some natural weakening of the conclusions.

 First we observe that \Cref{T:SonCP-T1} fails if one replaces the hypothesis 
 that some positive tensor power of $\Omega^1_{\mathcal X}(\log \mathbf \Delta)$ contains sheaf with big determinant, with the condition that some positive tensor power contains a sheaf with pseudo-effective determinant:
 
 \begin{exa}
For any elliptic curve $E$, one has that $\Omega^1_E$ contains a pseudo-effective line bundle, $\mathcal O_E$, but $K_E$ is not big.  
\end{exa}

\begin{rem} 
Despite the last example, we do have the following.  For a smooth projective curve $C$, if $\Omega^1_C(\log \Delta) $ contains a pseudo-effective line bundle, then $K_C+\Delta$  and any (quotient of any) tensor power of $\Omega^1_C(\log \Delta) $ are  pseudo-effective. \end{rem}

From the previous remark, one could imagine a natural generalization of 
  \Cref{T:SonCP-T1} and \Cref{T:SonCP-T4}  under the weaker hypothesis that  some
tensor power of $\Omega^1_{\mathcal X}(\log \mathbf \Delta)$ 
 contains a subsheaf with pseudo-effective determinant.  For instance, one could hope that under that weaker hypothesis, \Cref{T:SonCP-T4} might still hold, and \Cref{T:SonCP-T1} might hold if one replaced the conclusion that $K_{\mathcal X}+\mathbf \Delta$ be big with the conclusion that it only be pseudo-effective.  The following example shows that this is not the case:

 \begin{exa}
On the product $X=C\times \mathbb P^1$ of a smooth projective 
curve $C$ of genus $g>0$ and $\mathbb P^1$, one has the (split) short exact sequence
$$
0\to \pi_1^*K_C\to \Omega^1_{X}\to \pi_2^* K_{\mathbb P^1}\to 0,
$$
so that $ \Omega^1_{X}$ contains a pseudo-effective line bundle.  But clearly $K_{X}=\pi_1^*K_C+\pi_2^*K_{\mathbb P^1}$ is not pseudo-effective,  and $ \Omega^1_{X}$ has a quotient, $\pi_2^*K_{\mathbb P^1}$,   that is not pseudo-effective.  
\end{exa}
 
 \begin{rem} Note that in the example above, $X$ is uniruled.  In fact, \cite[Thm.~0.1]{CPT11} states that if $X$ is a smooth projective variety that is not uniruled, then any torsion-free quotient of any positive tensor power of $\Omega^1_X$ has pseudo-effective determinant. 
\end{rem}

\begin{rem}
What
happens in the example above, following through the proofs of  \Cref{T:SonCP-T1} and \Cref{T:SonCP-T4},  is that the argument in the last step in the proof of \Cref{P:SonCP-P13} breaks down: when $\mathcal L$ is not big, it may
be that $j = 0$, and then the inductive argument does not work.  More precisely, in the example, following through the proofs of  \Cref{T:SonCP-T1} and \Cref{T:SonCP-T4}, consider the case where one takes  $\alpha$ to be the class of a fiber of the first projection.  Then $\mathcal F=\mathcal F_\Delta= \pi_1^*\mathcal T_C$, and $p=\pi_1:X=C\times\mathbb P^1\to C$ is the first projection (i.e., $Z=C$).      One certainly has the inequality \eqref{E:SonCP-E4.5}, i.e., as $K_{X/Z}=\pi_2^*K_{\mathbb P^1}$, $\Delta^{hor}=0$, and $R(p)=0$, one has  $-K_{X/Z}.\alpha=2>0$.  But clearly, the conclusion of \Cref{P:SonCP-P13}, that $K_{X/Z}$ be pseudo-effective, does not hold.  Where the proof goes wrong is as follows.  In our case, the line bundle  $\mathcal L$ in the proof is $\pi_1^*K_C$, which is pseudo-effective, but not big.  Everything goes through, up through \eqref{E:SonCP-EP13}. However, without the assumption that $\mathcal L$ be big, we cannot conclude that $j>0$.  In fact, since $\mathcal L_{X_z}= \mathcal O_{X_z}$, and $\Omega^1_{X_z}=\mathcal O_{\mathbb P^1}(-2)$, a non-zero morphism $\mathcal L_{X_z}\to (\Omega^1_{ X_z})^{\otimes j}$ exists only if  $j\le 0$.  In other words, $j=0$ in the proof, and one cannot then use induction.   
\end{rem}

\fi 

 \bibliographystyle{amsalpha}
 \bibliography{mhm_bib}

\end{document}